\documentclass[final]{siamart190516}
\usepackage[margin=1.5in]{geometry}

\usepackage{lipsum}
\usepackage{amsfonts}
\usepackage{cite}
\usepackage{amsmath,amssymb,amsfonts}
\usepackage{algorithmic}
\usepackage{graphicx}
\usepackage{textcomp}
\usepackage{color}
\usepackage{bbm}
\usepackage{dsfont}
\usepackage{makeidx}
\usepackage{bbold}

\usepackage{commath}
\usepackage{mathtools}
\usepackage[utf8]{inputenc}
\usepackage[english]{babel}

\usepackage{dsfont}
\usepackage{makeidx}
\usepackage{bbold}
\usepackage{latexsym,remreset}
\usepackage{multirow}
\usepackage{amsfonts}
\usepackage{pgfplots}
\usepackage{mathtools}
\usepackage{textcomp}
\usepackage{scrextend}
\usepackage{commath}
\usepackage{graphicx}
\usepackage{url}
\usepackage{enumerate}

\usepackage{caption}
\usepackage{subcaption}

\usepackage{cleveref}

\newsiamremark{remark}{Remark}

\def\rank{\mathop{\bf rank}}

\def \C {\mathbb{C}}
\def \S {\mathbb{S}}
\def \R {\mathbb{R}}
\def \IU {\frak{i}}

\author{Amin Gholami \and X. Andy Sun
\thanks{The authors are with the H. Milton Stewart School of Industrial and Systems Engineering, Georgia Institute of Technology, Atlanta, GA 30332 USA (e-mail: \email{a.gholami{@}gatech.edu}; \email{andy.sun{@}isye.gatech.edu}).}
}
\usepackage{amsopn}

\begin{document}

\title{The Impact of Damping in Second-Order Dynamical Systems with Applications to Power Grid Stability}

\maketitle

\begin{abstract}
We consider a broad class of second-order dynamical systems and study the impact of damping as a system parameter on the stability, hyperbolicity, and bifurcation in such systems.
%
We prove a monotonic effect of damping on the hyperbolicity of the equilibrium points of the corresponding first-order system. This provides a rigorous formulation and theoretical justification for the intuitive notion that damping increases stability. To establish this result, we prove a matrix perturbation result for complex symmetric matrices with positive semidefinite perturbations to their imaginary parts, which may be of independent interest. 
Furthermore, we establish necessary and sufficient conditions for the breakdown of hyperbolicity of the first-order system under damping variations in terms of observability of a pair of matrices relating damping, inertia, and Jacobian matrices, and propose sufficient conditions for Hopf bifurcation resulting from such hyperbolicity breakdown.
The developed theory has significant applications in the stability of electric power systems, which are one of the most complex and important engineering systems. In particular, we characterize the impact of damping on the hyperbolicity of the swing equation model which is the fundamental dynamical model of power systems, and demonstrate Hopf bifurcations resulting from damping variations. 
%
%
%
\end{abstract}

\begin{keywords}
Second-order dynamical systems, electric power networks, swing equation, ordinary differential equations, damping, hyperbolicity.
\end{keywords}


\section{Introduction}
\label{sec:introduction}

In this paper, we study a class of second-order ordinary differential equations (ODEs) of the form
\begin{align} \label{eq: nonlinear ode}
M \ddot{x} + D \dot{x} + f(x) = 0,
\end{align}
and its corresponding first-order system
\begin{align} \label{eq: nonlinear ode 1 order_intro}
\begin{bmatrix} \dot{x} \\ \dot{y}
\end{bmatrix}
= 
\begin{bmatrix}
0 & I \\
0 & -M^{-1}D
\end{bmatrix}
\begin{bmatrix} {x} \\ {y}
\end{bmatrix}
- M^{-1} 
\begin{bmatrix} 0 \\ f(x)
\end{bmatrix},
\end{align}
where $f:\mathbb{R}^n\to \mathbb{R}^n$ is a continuously differentiable function, the dot denotes differentiation with respect to the independent variable $t\ge0$, the dependent variable $x\in\mathbb{R}^n$ is a vector of state variables, and the coefficients $M\in\S^n$ and $D\in\S^n$ are constant $n\times n$ real symmetric matrices.
We refer to $M$ and $D$ as the inertia and damping matrices, respectively. 
We restrict our attention to the case where $M$ is nonsingular, thereby avoiding differential algebraic equations, and $D \in \S^n_+$ is positive semi-definite (PSD). We also investigate and discuss the case where $M$ and $D$ are not symmetric.  
%


An important example of \eqref{eq: nonlinear ode} is an electric power system with the set of interconnected generators $\mathcal{N}=\{1,\cdots,n\}, n\in\mathbb{N}$ characterized by the second-order system
\begin{align} \label{eq: 2nd order swing-intro}
		\frac{m_j}{\omega_s} \ddot{\delta}_j(t)+ \frac{d_j}{\omega_s} {\dot{\delta}}_j(t) = P_{m_j} - \sum \limits_{k = 1}^n { V_j  V_k Y_{jk} \cos \left( \theta _{jk} - \delta _j + \delta _k \right)} && \forall j \in \mathcal{N},
\end{align}
where $\delta\in\R^n$ is the vector of state variables. The inertia and damping matrices in this case
are $M= \frac{1}{\omega_s}  \mathbf{diag}(m_1,\cdots,m_n)$ and $D=\frac{1}{\omega_s}\mathbf{diag}(d_1,\cdots,d_n)$. 
System \eqref{eq: 2nd order swing-intro}, which is  known as the \emph{swing equations}, describes the nonlinear dynamical relation between the power output and voltage angle of generators \cite{1994-kundur-stability, 2020-fast-certificate}.
%
%
The first-order system associated with swing equations is also of the form \eqref{eq: nonlinear ode 1 order_intro}, i.e.,
%
\begin{subequations} \label{eq: swing equations -intro}
	\begin{align}
		& \dot{\delta}_j(t) = \omega_j(t) && \forall j \in \mathcal{N},  \label{eq: swing equations a-intro}\\
		&  \frac{m_j}{\omega_s} \dot{\omega}_j(t)+ \frac{d_j}{\omega_s} \omega_j(t) = P_{m_j} - \sum \limits_{k = 1}^n { V_j  V_k Y_{jk} \cos \left( \theta _{jk} - \delta _j + \delta _k \right)} && \forall j \in \mathcal{N}, \label{eq: swing equations b-intro}
	\end{align}
\end{subequations}
where $(\delta,\omega)\in\R^{n+n}$ is the vector of state variables. Note that each generator $j$ is a second-order oscillator, which is coupled to other generators through the cosine term in \cref{eq: swing equations b-intro} and the admittance $Y_{jk}$ encodes the graph structure of the power grid (see \Cref{Sec: Power System Model and Its Properties} for full details on swing equations).
%

Among the various aspects of model \eqref{eq: nonlinear ode}, the impact of damping matrix $D$ on the stability properties of the model is one of the most intriguing topics \cite{1980-Miller-Asymptotic,2013-adhikari-structural-book, 2017-koerts-second-order}. Moreover, better understanding of the damping impact in swing equations \eqref{eq: 2nd order swing-intro} is of particular importance to the stability analysis of electric power systems \cite{2012-Dorfler-synchronization}.  
%
Undamped modes and oscillations are the root causes of several blackouts, such as the WECC blackout on August 10, 1996 \cite{1996-blackout} as well as the more recent events such as the forced oscillation event on January 11, 2019 \cite{nerc2019} in the Eastern Interconnection of the U.S. power grid. In order to maintain system stability in the wake of unexpected equipment failures, many control actions taken by power system operators are directly or indirectly targeted at changing the effective damping of system \eqref{eq: 2nd order swing-intro} \cite{1994-kundur-stability,2019-Patrick-koorehdavoudi-input,2012-aminifar-wide-area-damping}. In this context, an important question is \emph{how the stability properties of power system equilibrium points change as the damping of the system changes}. Our main motivation is to rigorously address this question for the general model \eqref{eq: nonlinear ode} and show its applications in power system model \eqref{eq: 2nd order swing-intro}.

\subsection{Literature Review}
The dynamical model \eqref{eq: nonlinear ode} has been of interest to many researchers who have studied necessary and sufficient conditions for its local stability \cite{1980-skar-nontrivial-transfer-conductance1, 2020-fast-certificate} or characterization of its stability regions \cite{1988-chiang-stability-of-nonlinear}. 
When $f(x)$ is a linear function, this model coincides with the model of $n$-degree-of-freedom viscously damped vibration systems which are also extensively studied in the structural dynamics literature \cite{1984-laub-controllability,2010-Ma-decoupling,2013-adhikari-structural-book}.
%
Equation \eqref{eq: nonlinear ode} is also the cornerstone of studying many physical and engineering systems such as an $n$-generator electric power system \cite{2012-Dorfler-synchronization}, an $n$-degree-of-freedom rigid body \cite{1988-chiang-stability-of-nonlinear}, and a system of $n$ coupled oscillators \cite{2011-dorfler-critical-coupling, 2012-Dorfler-synchronization, 2005-kuramoto-review}, in particular Kuramoto oscillators with inertia \cite{2016-Igor-Inertia-Kuramoto,2020-Igor-Inertia-Kuramoto}.

Regarding damping effects in power systems, the results are sporadic and mostly based on empirical studies of small scale power systems. For example, it is known that the lossless swing equations (i.e., when the transfer
conductances of power grid are zero, which corresponds to $\nabla f(x)$ in \eqref{eq: nonlinear ode} being a real symmetric matrix for all $x$) have no periodic solutions, provided that all generators have a positive damping value \cite{1982-Arapostathis-global-analysis-periodicSol}. It is also shown by numerical simulation that subcritical and supercritical Hopf bifurcations, and as a consequence, the emergence of periodic solutions, could happen if the swing equations of a two-generator network are augmented to include any of the following four features:  variable damping, frequency-dependent electrical torque, lossy transmission lines, and excitation control \cite{1981-Abed-Oscillations, 1986-Alexander-Oscillatory-Solutions}. Hopf bifurcation is also demonstrated in a three-generator undamped system as the load of the system changes \cite{1989-Kwatny-Energy-Analysis-Flutter}, where several energy functions for such undamped lossy swing equations in the neighborhood of points of Hopf bifurcation are developed to help characterize Hopf bifurcation in terms of energy properties. Furthermore, a frequency domain analysis to identify the stability of the periodic orbits created by a Hopf bifurcation is presented in \cite{1990-Kwatny-Frequency-Analysis_Hopf}.
The existence and the properties of limit cycles in power systems with higher-order models are also numerically analyzed in  \cite{2007-Hiskens-Limit-cycle2,2005-Hiskens-limit-cycle1}.

Another set of literature relevant to our work studies the role of power system parameters in the stability of its equilibrium points. For instance, the work presented in \cite{2019-Patrick-koorehdavoudi-input} examines the dependence of the transfer functions on the system parameters in the swing equation model. In \cite{2017-paganini-global}, the role of inertia in the frequency response of the system is studied. Moreover, it is shown how different dynamical models can lead to different conclusions. Finally, the works on frequency and transient stabilities in power systems \cite{2013-Lee-Power-dynamics-stochastic,2014-Low-NaLi-stability,2016-Vu-framework,2005-ortega-transient,2018-Dorfler-robust} are conceptually related to our work.
%



\subsection{Contributions}
This paper presents a thorough theoretical analysis of the role of damping in the stability of model \eqref{eq: nonlinear ode}-\eqref{eq: nonlinear ode 1 order_intro}. Our results provide rigorous formulation and theoretical justification for the intuitive notion that damping increases stability. The results also characterize the hyperbolicity and Hopf bifurcation of an equilibrium point of \eqref{eq: nonlinear ode 1 order_intro} through the inertia $M$, damping $D$, and  Jacobian $\nabla f$ matrices. These general results are applied to swing equations \eqref{eq: 2nd order swing-intro} to provide new insights into the damping effects on the stability of power grids.
%

The contributions and main results of this paper are summarized below.
\begin{enumerate}
    \item We show that increasing damping has a monotonic effect on the stability of equilibrium points in a large class of ODEs of the form \eqref{eq: nonlinear ode} and \eqref{eq: nonlinear ode 1 order_intro}. In particular, we show that, when $M$ is nonsingular symmetric, $D$ is symmetric PSD, and $\nabla f(x_0)$ is symmetric at an equilibrium point $(x_0,0)$ of the first-order system \eqref{eq: nonlinear ode 1 order_intro}, if the damping matrix $D$ is perturbed to $D'$ which is more PSD than $D$, i.e. $D'-D\in\S^n_+$, then the set of eigenvalues of the Jacobian of \cref{eq: nonlinear ode 1 order_intro} at $(x_0,0)$ that have a zero real part will not enlarge as a set (\cref{thrm: monotonic damping behaviour}). We also show that these conditions on $M, D, \nabla f(x_0)$ cannot be relaxed. To establish this result, we prove that the rank of a complex symmetric matrix with PSD imaginary part does not decrease if its imaginary part is perturbed by a real symmetric PSD matrix (\cref{thrm: rank}), which may be of independent interest in the matrix perturbation theory.

    \item 
    We propose a necessary and sufficient condition for an equilibrium point $(x_0,0)$ of the first-order system \eqref{eq: nonlinear ode 1 order_intro} to be hyperbolic. Specifically, when $M$ and $\nabla f(x_0)$ are symmetric positive definite and $D$ is symmetric PSD, then $(x_0,0)$ is hyperbolic if and only if the pair ($M^{-1}\nabla f(x_0), M^{-1}D$) is observable (\cref{thrm: nec and suf for pure imaginary lossless}). We extend the necessary condition to the general case where $M, D, \nabla f(x_0)$ are not symmetric (\cref{thrm: nec and suf for pure imaginary lossy}). Moreover, we characterize a set of sufficient conditions for the occurrence of Hopf bifurcation, when the damping matrix varies as a smooth function of a one dimensional bifurcation parameter (\cref{coro: Hopf bifurcation} and \cref{thm: fold and Hopf bifurcation}).
    
    \item We show that the theoretical results have key applications in the stability of electric power systems. 
    We propose a set of necessary and sufficient conditions for breaking the hyperbolicity in lossless power systems (\cref{prop: nec and suf for pure imaginary lossless power system}). We prove that in a lossy system with two or three generators, as long as only one generator is undamped, any equilibrium point is hyperbolic (\cref{prop:hyperbolicity n2 n3}), and as soon as there are more than one undamped generator, a lossy system with any $n\ge 2$ generators may lose hyperbolicity at its equilibrium points (\cref{prop: non-hyper example}). 
    Finally, we perform bifurcation analysis to detect Hopf bifurcation and identify its type based on two interesting case studies.
 \end{enumerate}

\subsection{Organization}
The rest of our paper is organized as follows. \Cref{sec: Background} introduces some notation and provides the problem statement. In \Cref{Sec: Monotonic Behavior of Damping}, we rigorously prove that damping has a monotonic effect on the local stability of a large class of ODEs. \Cref{Sec: Impact of Damping in Hopf Bifurcation} further investigates the impact of damping on hyperbolicity and bifurcation and presents a set of necessary and/or sufficient conditions for breaking the hyperbolicity and occurrence of bifurcations. \Cref{Sec: Power System Model and Its Properties} introduces the power system model (i.e., swing equations), provides a graph-theoretic interpretation of the system, and analyzes the practical applications of our theoretical results in power systems. \Cref{Sec: Computational Experiments} further illustrates the developed theoretical results through numerical examples, and finally, the paper concludes with \Cref{Sec: Conclusions}.
\section{Background} \label{sec: Background}
\subsection{Notations}
We use $\mathbb{C_{-/+}}$ to denote the set of complex numbers with negative/positive real part, and $\mathbb{C}_{0}$ to denote the set of complex numbers with zero real part.  $\IU=\sqrt{-1}$ is the imaginary unit. If $A\in\mathbb{C}^{m\times n}$, the transpose of $A$ is denoted by $A^\top$, the real part of $A$ is denoted by $\mathrm{Re}(A)$, and the imaginary part of $A$ is denoted by $\mathrm{Im}(A)$. The conjugate transpose of $A$ is denoted by $A^*$ and defined by $A^* = \Bar{A}^\top$, in which $\Bar{A}$ is the entrywise conjugate.
The matrix $A\in\mathbb{C}^{n \times n}$ is said to be symmetric if $A^\top = A$, Hermitian if $A^* = A$, and unitary if $A^*A = I$. The spectrum of a matrix $A\in\mathbb{R}^{n\times n}$ is denoted by $\sigma(A)$. We use $\mathbb{S}^n$ to denote the set of real symmetric $n\times n$ matrices, $\mathbb{S}^n_+$ to denote the set of real symmetric PSD $n\times n$ matrices, and $\mathbb{S}^n_{++}$ to denote the set of real symmetric positive definite $n\times n$ matrices. For matrices $A$ and $B$, the relation $B \succeq A$ means that $A$ and $B$ are real symmetric matrices of the same size such that $B-A$ is PSD; we write $A \succeq 0$ to express the fact that $A$ is a real symmetric PSD matrix. Strict version $B \succ A$ of $B \succeq A$ means that $B-A$ is real symmetric positive definite, and $A\succ 0$ means that $A$ is real symmetric positive definite.  
%

\subsection{Problem Statement} \label{subsec: Problem Statement}
Consider the second-order dynamical system \eqref{eq: nonlinear ode}.
The smoothness (continuous differentiability) of $f$ is a sufficient condition for the existence and uniqueness of solution. 
%
%
%
We transform \eqref{eq: nonlinear ode} into a system of $2n$ first-order ODEs of the form
\begin{align} \label{eq: nonlinear ode 1 order}
\begin{bmatrix} \dot{x} \\ \dot{y}
\end{bmatrix}
= 
\begin{bmatrix}
0 & I \\
0 & -M^{-1}D
\end{bmatrix}
\begin{bmatrix} {x} \\ {y}
\end{bmatrix}
- M^{-1} 
\begin{bmatrix} 0 \\ f(x)
\end{bmatrix}.
\end{align}
If $f(x_0)=0$ for some $x_0\in\mathbb{R}^n$, then $(x_0,0)\in\mathbb{R}^{n+n}$ is called an equilibrium point. The stability of such equilibrium points can be revealed by the spectrum of the Jacobian of the $2n$-dimensional vector field in \eqref{eq: nonlinear ode 1 order} evaluated at the equilibrium point. Note that $f:\R^n\to\R^n$ is a vector-valued function, and its derivative at any point $x\in\R^n$ is referred to as the Jacobian of $f$ and denoted by $\nabla f(x)\in\R^{n\times n}$. This Jacobian of $f$ should not be confused with the Jacobian of the $2n$-dimensional vector field in right-hand side of \eqref{eq: nonlinear ode 1 order}, which is 
\begin{align}\label{eq: J general case}
J(x) := \begin{bmatrix}
0 & I \\
-M^{-1} \nabla f(x)  & - M^{-1}D \\
\end{bmatrix} \in\mathbb{R}^{2n\times 2n}.
\end{align} 
If the Jacobian $J$ at an equilibrium point $(x_0,0)\in\mathbb{R}^{n+n}$
has all its eigenvalues off the imaginary axis, then we say that $(x_0,0)$ is a \emph{hyperbolic} equilibrium point. 
An interesting feature of hyperbolic equilibrium points is that they are either unstable or asymptotically stable. Breaking the hyperbolicity (say due to changing a parameter of the system), leads to bifurcation.
As mentioned before, we restrict our attention to the case where inertia matrix $M$ is nonsingular. Instead, we scrutinize the case where damping matrix $D$ is not full rank, i.e., the system is partially damped. This is a feasible scenario in real-world physical systems \cite{2017-koerts-second-order}, and as will be shown, has important implications specially in power systems.
Now, it is natural to ask the following questions:
\begin{enumerate}[(i)]
	\item How does changing the damping matrix $D$ affect the stability and hyperbolicity of equilibrium points of system \cref{eq: nonlinear ode 1 order}? \label{Q1}
	\item What are the conditions on $D$ under which an equilibrium point is hyperbolic? \label{Q2}
	\item When we lose hyperbolicity due to changing $D$, what kind of bifurcation happens? \label{Q3}
\end{enumerate}
Note that in these questions, the inertia matrix $M$ is fixed, and the bifurcation parameter only affects the damping matrix $D$.
Questions \eqref{Q1}-\eqref{Q3} will be addressed in the following sections, but before that, we present \cref{lemma: relation between ev J and ev J11} \cite{2020-fast-certificate} which provides some intuition behind the role of different factors in the spectrum of the Jacobian matrix $J$.
Let us define the concept of matrix pencil \cite{2001-Tisseur-Pencil}. Consider $n\times n$ matrices $Q_0,Q_1,$ and $Q_2$. A quadratic matrix pencil is a matrix-valued function $P:\mathbb{C}\to\mathbb{R}^{n \times n}$ given by $\lambda \mapsto P(\lambda)$ such that  $P(\lambda) = \lambda^2Q_2 + \lambda Q_1 + Q_0$.
\begin{lemma} \label{lemma: relation between ev J and ev J11}
	For any $x\in\R^n$, $\lambda$ is an eigenvalue of $J(x)$ if and only if the quadratic matrix pencil $P(\lambda):= \lambda^2 M + \lambda D + \nabla f(x)$ is singular.
\end{lemma}
\begin{proof}
	For any $x\in\R^n$, let $\lambda$ be an eigenvalue of $J(x)$ and $( v , u )$ be the corresponding eigenvector. Then	
	\begin{align} \label{eq: J cha eq}
	\begin{bmatrix}
	0 & I \\
	-M^{-1} \nabla f(x)     &     - M^{-1}D \\
	\end{bmatrix}   \begin{bmatrix} v \\u  \end{bmatrix}  = \lambda   \begin{bmatrix} v \\u  \end{bmatrix},
	\end{align}
	which implies that $ u = \lambda v$ and $-M^{-1} \nabla f(x) v - M^{-1}D u  = \lambda u$. Substituting the first equality into the second one, we get
	\begin{align} 
	\left( \nabla f(x) + \lambda D +  \lambda^2 M \right) v = 0. \label{eq: quadratic matrix pencil}
	\end{align}  
	Since $v \not = 0$ (otherwise $u = \lambda \times 0 = 0 $ which is a contradiction), equation \eqref{eq: quadratic matrix pencil} implies that the matrix pencil $P(\lambda)= \lambda^2 M + \lambda D + \nabla f(x)$ is singular.
	
	Conversely, for any $x\in\R^n$,
	suppose there exists $\lambda \in \mathbb{C}$ such that  $P(\lambda)= \lambda^2 M + \lambda D + \nabla f(x)$ is singular. Choose a nonzero $v \in   \ker(P(\lambda))$ and let $ u := \lambda v$. 
	Accordingly, the characteristic equation \eqref{eq: J cha eq} holds, and consequently, $\lambda$ is an eigenvalue of $J(x)$.
\end{proof}

To give some intution, let us pre-multiply \eqref{eq: quadratic matrix pencil} by $v^*$ to get the quadratic equation
\begin{align} 
v^* \nabla f(x) v + \lambda v^*Dv +  \lambda^2 v^*Mv  = 0, \label{eq: quadratic matrix pencil equation}
\end{align}
which has roots
\begin{align} 
\lambda_{\pm} = \frac{-v^*Dv \pm \sqrt{(v^*Dv)^2-4(v^*Mv)(v^* \nabla f(x) v)}}{2v^*Mv}.
\label{eq: quadratic matrix pencil equation roots}
\end{align}
Equation \eqref{eq: quadratic matrix pencil equation roots} provides some insights into the impact of matrices $D$, $M$, and $\nabla f(x)$ on the eigenvalues of $J$. For instance, when $D\succeq0$, it seems that increasing the damping matrix $D$ (i.e., replacing $D$ with $\hat{D}$, where $\hat{D}\succeq D$) will lead to more over-damped eigenvalues. However, this argument is not quite compelling because by changing $D$, the eigenvector $v$ would also change.     
Although several researchers have mentioned such arguments about the impact of damping \cite{2010-Ma-decoupling}, to the best of our knowledge, this impact has not been studied in the literature in a rigorous fashion. We will discuss this impact in the next section.
\section{Monotonic Effect of Damping}
\label{Sec: Monotonic Behavior of Damping}
In this section, we analytically examine the role of damping matrix $D$ in the stability of system \eqref{eq: nonlinear ode}. Specifically, we answer the following question:
let System-I and System-II be two second-order dynamical systems \eqref{eq: nonlinear ode} with partial damping matrices $D_I\succeq 0$ and $D_{II}\succeq 0$, respectively. Suppose the two systems are identical in other parameters (i.e., everything except their dampings) and $(x_0,0)\in\mathbb{R}^{2n}$ is an equilibrium point for both systems. Observe that changing the damping of system \eqref{eq: nonlinear ode} does not change the equilibrium points. Here, we focus on the case where $M$ and $L:=\nabla f(x_0)$ are symmetric (these are reasonable assumptions in many dynamical systems such as power systems). Now, if System-I is asymptotically stable, what kind of relationship between $D_I$ and $D_{II}$ will ensure that System-II is also asymptotically stable?   
%
This question has important practical consequences. 
For instance, the answer to this question will illustrate how changing the damping coefficients of generators (or equivalently, the corresponding controller parameters of inverter-based resources) in power systems will affect the stability of equilibrium points. Moreover, this question is closely intertwined with a problem in matrix perturbation theory, namely 
given a complex symmetric matrix with PSD imaginary part, how does a PSD perturbation of its imaginary part affect the rank of the matrix? We answer the matrix perturbation question in \Cref{thrm: rank}, which requires \Cref{lemma: Autonne} to \Cref{lemma: rank principal} and \Cref{Prop: diagonal complex update nonsingularity}. Finally, the main result about the monotonic effect of damping is proved in \Cref{thrm: monotonic damping behaviour}.

The following lemma on Autonne-Takagi factorization is useful.

\begin{lemma}[Autonne-Takagi factorization] \label{lemma: Autonne}
Let $S\in \mathbb{C}^{n\times n}$ be a complex matrix. Then $S^\top=S$ if and only if there is a unitary $U\in \mathbb{C}^{n\times n}$ and a nonnegative diagonal matrix $\Sigma\in \mathbb{R}^{n\times n}$ such that $S=U\Sigma U^\top$. The diagonal entries of $\Sigma$ are the singular values of $S$.
\end{lemma}	
\begin{proof}
See e.g. \cite[Chapter 4]{2013-Horn-matrix-analysis}.
\end{proof}

We also need the following lemmas to derive our main results. \Cref{lemma: definiteness inverse imaginary part} generalizes a simple fact about complex numbers to complex symmetric matrices: a complex scalar $z\in\mathbb{C}, z\ne0$ has a nonnegative imaginary part if and only if $z^{-1}$ has a nonpositive imaginary part.
\begin{lemma} \label{lemma: definiteness inverse imaginary part}
	Let $S\in \mathbb{C}^{n\times n}$ be a nonsingular complex symmetric matrix. Then $\mathrm{Im}(S) \succeq 0$ if and only if $\mathrm{Im}(S^{-1})\preceq 0$. 
\end{lemma}	
\begin{proof}
Since $S$ is nonsingular complex symmetric, by Autonne-Takagi factorization, there exists a unitary matrix $U$ and a diagonal positive definite matrix $\Sigma$ such that $S = U\Sigma U^\top$. The inverse $S^{-1}$ is given by $S^{-1} = \bar{U}\Sigma^{-1}U^*$. The imaginary parts of $S$ and $S^{-1}$ are 
	\begin{align*}
	& \mathrm{Im}(S) = -\frac{1}{2} \IU (U\Sigma U^\top - \bar{U}\Sigma U^*), \\
	& \mathrm{Im}(S^{-1}) = -\frac{1}{2} \IU (\bar{U}\Sigma^{-1} U^* - U\Sigma^{-1} U^\top).
	\end{align*}
	The real symmetric matrix $2\mathrm{Im}(S^{-1})=\IU (U\Sigma^{-1} U^\top-\bar{U}\Sigma^{-1} U^*)$ is unitarily similar to the Hermitian matrix $\IU(\Sigma^{-1}U^\top U - U^*\bar{U}\Sigma^{-1})$ as
	\begin{align*}
	 U^*(2\mathrm{Im}(S^{-1}))U &= U^*(\IU(U\Sigma^{-1} U^\top - \bar{U}\Sigma^{-1} U^*))U \\
	 & = \IU (\Sigma^{-1}U^\top U - U^*\bar{U}\Sigma^{-1}), 
	\end{align*} 
	and is *-congruent to $\IU (U^\top U \Sigma - \Sigma U^*\bar{U})$ as
	\begin{align*}
     \Sigma U^*(2\mathrm{Im}(S^{-1}))U\Sigma &= \IU\Sigma(\Sigma^{-1}U^\top U - U^*\bar{U}\Sigma^{-1})\Sigma \\
	& = \IU (U^\top U \Sigma - \Sigma U^*\bar{U}). 
	\end{align*}
	Note that the latter transformation is a *-congruence because $U\Sigma$ is nonsingular but not necessarily unitary. Hence, $2\mathrm{Im}(S^{-1})$ has the same eigenvalues as $\IU(\Sigma^{-1}U^\top U - U^*\bar{U}\Sigma^{-1})$ and has the same inertia as $\IU(U^\top U \Sigma - \Sigma U^*\bar{U})$ by Sylvester's law of inertia. 
	Furthermore, since $U^\top U$ is unitary and $\overline{U^\top U}=U^*\bar{U}$, then
	\begin{align*}
	\IU(U^\top U \Sigma - \Sigma U^*\bar{U}) = (U^\top U)(\IU(\Sigma U^\top U - U^*\bar{U}\Sigma))(U^*\bar{U}),
	\end{align*}
	which implies that $\IU(U^\top U \Sigma - \Sigma U^*\bar{U})$ has the same eigenvalues as $\IU(\Sigma U^\top U - U^*\bar{U}\Sigma)$. Furthermore, since 
	\begin{align*}
	U(\IU(\Sigma U^\top U - U^*\bar{U}\Sigma))U^* = \IU(U\Sigma U^\top - \bar{U}\Sigma U^*) = -2\mathrm{Im}(S),
	\end{align*}
	$\mathrm{Im}(S^{-1})$ and $-\mathrm{Im}(S)$ have the same inertia, i.e., they have the same number of positive eigenvalues and the same number of negative eigenvalues. Therefore, $\mathrm{Im}(S)\succeq 0$ if and only if all eigenvalues of $\mathrm{Im}(S^{-1})$ are nonpositive, that is, if and only if $\mathrm{Im}(S^{-1})\preceq0$.
\end{proof}

\Cref{lemma: rank-one imaginary PSD update} shows how rank-one perturbation to the imaginary part of a nonsingular complex matrix preserves its nonsingularity.
\begin{lemma} \label{lemma: rank-one imaginary PSD update}
	Let $S\in\mathbb{C}^{n\times n}$ be a nonsingular complex symmetric matrix. If $\mathrm{Im}(S)\succeq 0 $, then $S+\IU vv^\top$ is nonsingular for any real vector $v\in\mathbb{R}^n$.
\end{lemma}
\begin{proof}
	We use Cauchy's formula for the determinant of a rank-one perturbation \cite{2013-Horn-matrix-analysis}:
	\begin{align*}
	\det(S + \IU vv^\top) &= \det(S) + \IU v^\top\mathrm{adj}(S)v\\
	&= \det(S) + \IU v^\top S^{-1}v \det(S) \\
	&= \det(S)(1 + \IU v^\top S^{-1}v) \\
	&= \det(S)(1 - v^\top\mathrm{Im}(S^{-1})v + \IU  v^\top\mathrm{Re}(S^{-1})v),
	\end{align*}
	where $\mathrm{adj}(S)$ is the adjugate of $S$, which satisfies $\mathrm{adj}(S)=(\det(S))S^{-1}$. 
	Since $\det(S)\ne0$, we only need to prove that the complex scalar $z:=(1 - v^\top\mathrm{Im}(S^{-1})v + \IU  v^\top\mathrm{Re}(S^{-1})v)$ is nonzero for any $v\in\mathbb{R}^n$.
	By \cref{lemma: definiteness inverse imaginary part}, $\mathrm{Im}(S^{-1}) \preceq 0$, thus $\mathrm{Re}(z)=1-v^\top\mathrm{Im}(S^{-1})v \ge 1$. This proves that any rank-one update on the imaginary part of $S$ is nonsingular.
\end{proof}

Now, we extend \cref{lemma: rank-one imaginary PSD update} to the case where the perturbation is a general real PSD matrix.
\begin{proposition} \label{Prop: diagonal complex update nonsingularity}
	Let $S\in\mathbb{C}^{n\times n}$ be a nonsingular complex symmetric matrix with $\mathrm{Im}(S)\succeq 0 $. Then, for any real PSD matrix $E\in\mathbb{S}^{n}_+$, $S+\IU E$ is nonsingular.
\end{proposition}
\begin{proof}
	Since $E$ is a real PSD matrix, its eigendecomposition gives $E=\sum_{\ell=1}^n v_\ell v_\ell^\top$, where $v_\ell$ is an eigenvector scaled by the $\ell$-th eigenvalue of $E$. Now, we need to show that $S+\sum_{\ell=1}^n \IU  v_\ell v_\ell^\top$ is nonsingular. According to \cref{lemma: rank-one imaginary PSD update}, $\tilde{S}_\ell:=S + \IU  v_\ell v_\ell^\top$ is nonsingular for each $\ell \in \{1,\cdots,n\}$. Moreover, $\tilde{S}_\ell$ is a complex symmetric matrix with $\textrm{Im}(\tilde{S}_\ell)\succeq 0$. Therefore, \cref{lemma: rank-one imaginary PSD update} can be consecutively applied to conclude that $S+\IU E$ is nonsingular.
\end{proof}
\begin{remark}
	The assumption of $S$ being complex symmetric cannot be relaxed. For example, consider unsymmetric matrix
	\begin{align*}
	S = \begin{bmatrix}
	1+\IU  & \sqrt{2} \\ -\sqrt{2} & -1
	\end{bmatrix}, E = \begin{bmatrix} 0 & 0\\0 & 1\end{bmatrix}, S+\IU E = \begin{bmatrix}
	1+\IU  & \sqrt{2} \\ -\sqrt{2} & -1+\IU 
	\end{bmatrix}.
	\end{align*}
	Then, $\textrm{Im}(S)\succeq 0$, $\det(S)=1-\IU $, but $\det(S+\IU E)=0$. Likewise, the assumption of $E$ being real PSD cannot be relaxed.
\end{remark}
Before proceeding further with the analysis, let us recall the concept of principal submatrix. For $A\in\mathbb{C}^{n\times n}$ and $\alpha \subseteq \{1,\cdots,n\}$, the (sub)matrix of entries that lie in the rows and columns of $A$ indexed by $\alpha$ is called a principal submatrix of $A$ and is denoted by $A[\alpha]$. We also need \cref{lemma: rank principal} about rank principal matrices. In what follows, the direct sum of two matrices $A$ and $B$ is denoted by $A\oplus B$.
%
\begin{lemma} [rank principal matrices] \label{lemma: rank principal}
	Let $S\in \mathbb{C}^{n\times n}$ and suppose that $n>\rank(S)=r \ge 1$. If $S$ is similar to $B\oplus0_{n-r}$ (so
	$B \in  \mathbb{C}^{r\times r} $is nonsingular), then $S$ has a nonsingular $r$-by-$r$ principal submatrix, that is,
	$S$ is rank principal.
\end{lemma}
\begin{proof}
    See \cref{proof of lemma: rank principal}.
\end{proof}
%
%
%
Now we are ready to state our main matrix perturbation result.
\begin{theorem} \label{thrm: rank}
Suppose $A\in\mathbb{S}^{n}$ is a real symmetric matrix and $D\in\mathbb{S}^{n}_+$ and $E\in\mathbb{S}^{n}_+$ are real symmetric PSD matrices. 
Then $\rank(A+\IU D)\le \rank(A+\IU D+\IU E)$. 
\end{theorem}
\begin{proof}
Define $r:=\rank(A+\IU D)$ and note that if $r=0$, i.e., $A+\IU D$ is the zero matrix, then the rank inequality holds trivially. If $r\ge1$, the following two cases are possible.


	For $r=n$ : in this case $S:=A+\IU D$ is a nonsingular complex symmetric matrix with $\mathrm{Im}(S)\succeq 0 $, and according to \cref{Prop: diagonal complex update nonsingularity}, $A+\IU D+\IU E$ is also nonsingular, i.e., $\rank(A+\IU D+\IU E)=n$. Thus, in this case, the rank inequality $\rank(A+\IU D)\le \rank(A+\IU D+\IU E)$ holds with equality.

	 For $1\le r < n$: since $A+\IU D$ is complex symmetric, using Autonne-Takagi factorization in \cref{lemma: Autonne}, $A+\IU D = U\Sigma U^\top$ for some unitary matrix $U$ and a diagonal real PSD matrix $\Sigma$. Moreover, $r=\rank(A+\IU D)$ will be equal to the number of positive diagonal entries of $\Sigma$. In this case, $A+\IU D$ is unitarily similar to $\Sigma=B\oplus0_{n-r}$, for some nonsingular diagonal $B\in\mathbb{R}^{r\times r}$. According to \cref{lemma: rank principal}, there exists a principal submatrix of $A+\IU D$ with size $r$ that is nonsingular, that is, there exists an index set $\alpha \subseteq \{1,\cdots,n\}$ with $\textrm{card}(\alpha)=r$ such that $A[\alpha]+\IU D[\alpha]$ is nonsingular. Note that $A[\alpha]+\IU D[\alpha]$ is also complex symmetric. Now, using the same index set $\alpha$ of  rows and columns, we select the principal submatrix $E[\alpha]$ of $E$. Recall that if a matrix is PSD then all its principal submatrices are also PSD. Therefore, $D[\alpha] \succeq 0$ and $E[\alpha] \succeq 0$.
	 %
    %
     Using the same argument as in the previous case of this proof, we have $\rank(A[\alpha]+\IU D[\alpha]) = \rank(A[\alpha]+\IU D[\alpha]+\IU E[\alpha])=r$. 
     On the one hand, according to our assumption, we have $\rank(A+\IU D)=\rank(A[\alpha]+\IU D[\alpha])=r$. On the other hand, we have
	\begin{align} \label{eq: rank ineq of rank thrm}
		 \rank(A+\IU D+\IU E) \ge \rank(A[\alpha]+\IU D[\alpha]+\IU E[\alpha]) = r = \rank(A+\IU D).
	\end{align}
	Note that the inequality in \eqref{eq: rank ineq of rank thrm} holds because the rank of a principal submatrix is always less than or equal to the rank of the matrix itself. In other words, by adding more columns and rows to a (sub)matrix, the existing linearly independent rows and columns will remain linearly independent. Therefore, the rank inequality $\rank(A+\IU D)\le \rank(A+\IU D+\IU E)$ also holds in this case.
\end{proof}

We now use \cref{thrm: rank} to answer the question of how damping affects stability. In particular, \cref{thrm: monotonic damping behaviour} shows a monotonic effect of damping on system stability. Namely, when $\nabla f(x_0)$ is symmetric at an equilibrium point $(x_0,0)$, the set of eigenvalues of the Jacobian $J(x_0)$ that lie on the imaginary axis does not enlarge, as the damping matrix $D$ becomes more positive semidefinite. 
\begin{theorem}[Monotonicity of imaginary eigenvalues in response to damping] \label{thrm: monotonic damping behaviour}
    Consider the following two systems,
    \begin{align} 
         \quad & M \ddot{x} + D_{I}  \dot{x} + f(x) = 0, \tag{System-I} \label{System-I} \\
         \quad & M \ddot{x} + D_{II} \dot{x} + f(x) = 0, \tag{System-II} \label{System-II}
    \end{align}
    where $M\in\mathbb{S}^n$ is nonsingular and $D_{I},D_{II} \in \mathbb{S}^n_+$. Suppose $(x_0,0)$ is an equilibrium point of the corresponding first-order systems defined in \eqref{eq: nonlinear ode 1 order}. Assume $L:=\nabla f(x_0)\in\mathbb{S}^n$. Denote $J_I, J_{II}$ as the associated Jacobian matrices at $x_0$ as defined in \eqref{eq: J general case}. Furthermore, 
    %
    let $\mathcal{C}_{I} \subseteq \mathbb{C}_{0} $ (resp. $\mathcal{C}_{II} \subseteq \mathbb{C}_{0}$) be the set of eigenvalues of $J_{I}$ (resp. $J_{II}$) with a zero real part, which may be an empty set. Then the sets $C_I, C_{II}$ of eigenvalues on the imaginary axis satisfy the following monotonicity property, 
    \begin{align}
      D_{II} \succeq D_{I} \implies \mathcal{C}_{II}  \subseteq \mathcal{C}_{I}.  
    \end{align}
\end{theorem}
\begin{proof}
Recall the Jacobian matrices are defined as
\begin{align*}
 J_I = \begin{bmatrix}
0 & I \\
- M^{-1} L     & -M^{-1} D_{I}  \\
\end{bmatrix},
J_{II} = \begin{bmatrix}
0 & I \\
- M^{-1} L   &   - M^{-1} D_{II} \\
\end{bmatrix},
\end{align*}
at an equilibrium point $(x_0,0)$.
According to \cref{lemma: relation between ev J and ev J11}, $\IU \beta\in\mathcal{C}_{I}$ if and only if the quadratic matrix pencil $P_I(\IU \beta):= (L - \beta^2 M) + \IU \beta D_I$ is singular. The same argument holds for $\mathcal{C}_{II}$. Since $J_I$ and $J_{II}$ are real matrices, their complex eigenvalues will always occur in complex conjugate pairs. Therefore, without loss of generality, we assume $\beta\ge0$. Note that for any $\beta\ge0$ such that $\IU \beta\notin\mathcal{C}_{I}$ the pencil $P_I(\IU \beta)= (L - \beta^2 M) + \IU \beta D_I$ is nonsingular. 
Moreover, $(L-\beta^2M)\in\mathbb{S}^{n}$ is a real symmetric matrix and $\beta D_I\in\mathbb{S}^{n}_+$ is a real PSD matrix. According to \cref{thrm: rank},  
\begin{align*}
r & =  \rank(L-\beta^2 M +\IU \beta D_I) \\
  &\le \rank(L-\beta^2 M +\IU \beta D_I + \IU \beta(D_{II}-D_{I})) \\
  & = \rank(L-\beta^2  M +\IU \beta D_{II}),
\end{align*}
consequently, $P_{II}(\IU \beta)=L-\beta^2M+\IU \beta D_{II}$ is also nonsingular and $\IU \beta\notin\mathcal{C}_{II}$. This implies that $\mathcal{C}_{II}  \subseteq \mathcal{C}_{I}$ and completes the proof.
\end{proof}
\begin{remark}
    In the above theorem, the assumption of $L=\nabla f(x_0)$ being symmetric cannot be relaxed. For example, consider 
	\begin{align*}
	f(x_1,x_2) = \begin{bmatrix}
	2x_1 + \sqrt{2} x_2 \\ -\sqrt{2} x_1
	\end{bmatrix},
	D_{I} = \begin{bmatrix} 1 & 0\\0 & 0\end{bmatrix}, D_{II} = \begin{bmatrix} 1 & 0\\0 & 1\end{bmatrix}, M = \begin{bmatrix} 1 & 0\\0 & 1\end{bmatrix}.
	\end{align*}
	Here, the origin is the equilibrium point of the corresponding first-order systems, and $L=\nabla f(0,0)$ is not symmetric. The set of eigenvalues with zero real part in \eqref{System-I} and \eqref{System-II} are $\mathcal{C}_I=\emptyset$ and $\mathcal{C}_{II}=\{ \pm \IU  \}$. Accordingly, we have $D_{II}\succeq D_{I}$, but $\mathcal{C}_{II}  \not \subseteq \mathcal{C}_{I}$.
\end{remark}

\section{Impact of Damping on Hyperbolicity and Bifurcation}
\label{Sec: Impact of Damping in Hopf Bifurcation}
%
\subsection{Necessary and Sufficient Conditions for Breaking Hyperbolicity}
\label{subsec: Sufficient and Necessary Conditions}
We use the notion of observability from control theory to provide a necessary and sufficient condition for breaking the hyperbolicity of equilibrium points in system \eqref{eq: nonlinear ode 1 order} when the inertia, damping, and Jacobian of $f$ satisfy $M\in\S^{n}_{++}, D\in\S^n_+, \nabla f(x_0)\in\S^n_{++}$ at an equilibrium point $(x_0,0)$ (\cref{thrm: nec and suf for pure imaginary lossless}). We further provide a sufficient condition for the existence of purely imaginary eigenvalues in system \eqref{eq: nonlinear ode 1 order} when $M, D, \nabla f(x_0)$ are not symmetric (\cref{thrm: nec and suf for pure imaginary lossy}). Such conditions will pave the way for understanding Hopf bifurcations in these systems. Observability was first related to stability of second-order system \eqref{eq: nonlinear ode} by Skar \cite{1980-Skar-stability-thesis}.
\begin{definition}[observability] \label{def: Observability}
	Consider the matrices $A\in\mathbb{R}^{m\times m}$ and $B\in\mathbb{R}^{n\times m}$. The pair $(A,B)$ is observable if $Bx\ne 0$ for every right eigenvector $x$ of $A$, i.e.,
	\begin{align*}
	\forall \lambda \in \mathbb{C}, x \in \mathbb{C}^m, x\ne 0  \: \text{ s.t. } Ax = \lambda x \implies Bx \ne 0.
	\end{align*}
\end{definition}

We will show that the hyperbolicity of an equilibrium point $(x_0,0)$ of system \eqref{eq: nonlinear ode 1 order} is intertwined with the observability of the pair $(M^{-1}\nabla f(x_0),M^{-1}D)$. Our focus will remain on the role of the damping matrix $D\in\mathbb{S}^n_+$ in this matter. Note that if the damping matrix $D$ is nonsingular, the pair $(M^{-1}\nabla f(x_0),M^{-1}D)$ is always observable because the nullspace of $M^{-1}D$ is trivial. Furthermore, if the damping matrix $D$ is zero, the following lemma holds.
\begin{lemma} \label{lemma: undamped systems}
    In an undamped system (i.e., when $D=0$), for any $x\in\R^n$ the pair $(M^{-1}\nabla f(x),M^{-1}D)$ can never be observable. Moreover, for any $x\in\R^n$
    \begin{align*}
        \lambda \in\sigma(J(x))\iff \lambda^2\in\sigma(-M^{-1}\nabla f(x)).
    \end{align*}
\end{lemma}
\begin{proof}
The first statement is an immediate consequence of \cref{def: Observability} and the second one follows from \cref{lemma: relation between ev J and ev J11}.
\end{proof}

The next theorem yields a necessary and sufficient condition on the damping matrix $D$ for breaking the hyperbolicity of an equilibrium point.
%
%
\begin{theorem}[hyperbolicity in second-order systems: symmetric case] \label{thrm: nec and suf for pure imaginary lossless}
    Consider the second-order ODE system \eqref{eq: nonlinear ode} with inertia matrix $M\in\mathbb{S}^n_{++}$ and damping matrix $D \in\mathbb{S}^n_+$. Suppose $(x_0,0)\in \mathbb{R}^{n+n}$ is an equilibrium point of the corresponding first-order system \eqref{eq: nonlinear ode 1 order} with the Jacobian matrix $J\in\mathbb{R}^{2n\times 2n}$ defined in \eqref{eq: J general case} such that $L=\nabla f(x_0)\in \mathbb{S}^n_{++}$. Then, the equilibrium point $(x_0,0)$ is hyperbolic if and only if the pair $(M^{-1}L,M^{-1}D)$ is observable.
    %
    %
    %
        
\end{theorem}
\begin{proof}
    According to \cref{lemma: undamped systems}, if $D=0$, the pair $(M^{-1}L,M^{-1}D)$ can never be observable. 
	Moreover, $M^{-1}L=M^{-\frac{1}{2}}  \hat{L}  M^{\frac{1}{2}}$, where $\hat{L}:=M^{-\frac{1}{2}} L M^{-\frac{1}{2}}$. This implies that $M^{-1}L$ is similar to (and consequently has the same eigenvalues as) $\hat{L}$. 
 	Note that $\hat{L}$ is *-congruent to $L$. According to Sylvester’s law of inertia, $\hat{L}$ and $L$ have the same inertia. Since $L\in \mathbb{S}^n_{++}$, we conclude that $\hat{L}\in \mathbb{S}^n_{++}$. Therefore, the eigenvalues of $M^{-1}L$ are real and positive, i.e., $\sigma(M^{-1}L)\subseteq \mathbb{R_{++}}=\{\lambda\in\mathbb{R}:\lambda>0\}$.
    %
	Meanwhile, when $D=0$, we have $\mu\in\sigma(J)\iff\mu^2\in\sigma(-M^{-1}L)$, hence all eigenvalues of $J$ would have zero real parts, i.e., $\sigma(J)\subseteq\mathbb{C}_0$, and consequently, the theorem holds trivially. In the sequel, we assume that $D\not=0$.

	\emph{Necessity:}
	Suppose the pair $(M^{-1}L,M^{-1}D)$ is observable, but assume the equilibrium point is not hyperbolic, and let us lead this assumption to a contradiction. Since $L=\nabla f(x_0)$ is nonsingular, \cref{lemma: relation between ev J and ev J11} asserts that $0\not\in\sigma(J)$. Therefore, there must exist $\beta>0$ such that $\IU \beta \in \sigma(J)$. 
	By \cref{lemma: relation between ev J and ev J11}, $\IU\beta \in \sigma(J)$ if and only if the matrix pencil $(M^{-1}L + \IU\beta M^{-1}D - \beta^2 I)$ is singular: 
	\begin{align*} 
	& \det \left(  M^{-\frac{1}{2}}(M^{-\frac{1}{2}}LM^{-\frac{1}{2}} + \IU \beta  M^{-\frac{1}{2}}DM^{-\frac{1}{2}} - \beta^2 I) M^{\frac{1}{2}} \right) = 0,
	\end{align*}
	or equivalently, $\exists (x + \IU y) \ne 0$ such that $x,y\in\mathbb{R}^n$ and
	\begin{align} \label{eq: proof of parially damped hopf bifurcation}
	\notag & (M^{-\frac{1}{2}}LM^{-\frac{1}{2}} + \IU \beta  M^{-\frac{1}{2}} D M^{-\frac{1}{2}} - \beta^2 I)(x+\IU y) = 0 \\
	& \iff  \begin{cases}
	(M^{-\frac{1}{2}}LM^{-\frac{1}{2}}-\beta^2 I)x - \beta M^{-\frac{1}{2}}DM^{-\frac{1}{2}} y = 0, \\ 
	(M^{-\frac{1}{2}}LM^{-\frac{1}{2}}-\beta^2 I)y + \beta M^{-\frac{1}{2}}DM^{-\frac{1}{2}} x = 0.
	\end{cases}
	\end{align}
	Let $\hat{L}:=M^{-\frac{1}{2}}LM^{-\frac{1}{2}}$, $\hat{D}:=M^{-\frac{1}{2}}DM^{-\frac{1}{2}}$, and observe that
	\begin{align*}
	\begin{cases}
	y^\top (\hat{L}-\beta^2 I)x = y^\top (\beta \hat{D} y) = \beta y^\top \hat{D} y  \ge 0,
	\\
	x^\top (\hat{L}-\beta^2 I)y = x^\top (-\beta \hat{D} x) = -\beta x^\top \hat{D} x  \le 0,
	\end{cases}
	\end{align*}
    %
	where we have used the fact that $\hat{D}$ is *-congruent to $D$. According to Sylvester’s law of inertia, $\hat{D}$ and $D$ have the same inertia. Since $D\succeq 0$, we conclude that $\hat{D}\succeq0$.
	As $(\hat{L}-\beta^2 I)$ is symmetric, we have $x^\top (\hat{L}-\beta^2 I)y = y^\top (\hat{L}-\beta^2 I)x$. Therefore, we must have $x^\top \hat{D} x = y^\top \hat{D} y  =0$. Since $\hat{D}\succeq0$, we can infer that $x\in\ker(\hat{D})$ and $y\in\ker(\hat{D})$.
	%
	Now considering $\hat{D} y = 0$ and using the first equation in \eqref{eq: proof of parially damped hopf bifurcation} we get 
	\begin{align}
	( \hat{L}-\beta^2 I)x = 0 \iff M^{-\frac{1}{2}}LM^{-\frac{1}{2}} x=\beta^2 x,
	\end{align}
	multiplying both sides from left by $M^{-\frac{1}{2}}$ we get $M^{-1}L (M^{-\frac{1}{2}} x) = \beta^2 (M^{-\frac{1}{2}} x)$. Thus, $ \hat{x}:= M^{-\frac{1}{2}} x$ is an eigenvector of $M^{-1}L$. Moreover, we have 
	\begin{align*}
	    M^{-1}D \hat{x}  = M^{-1}DM^{-\frac{1}{2}} x = M^{-\frac{1}{2}} (  \hat{D}  x) = 0,    
	\end{align*}
	which means that the pair $(M^{-1}L,M^{-1}D)$ is not observable; we have arrived at the desired contradiction.

	\emph{Sufficiency:}
	%
	Suppose the equilibrium point is hyperbolic, but assume that the pair~ $\allowbreak (M^{-1}L, M^{-1}D)$ is not observable;  we will show that this assumption leads to a contradiction.	According to \cref{def: Observability}, $\exists \lambda \in \mathbb{C}, x \in \mathbb{C}^n , x\ne 0$ such that
	\begin{align} \label{eq: observability in lossless general}
	 M^{-1}Lx = \lambda x \text{ and } M^{-1}Dx = 0.
	\end{align}
	We make the following two observations.
	Firstly, as it is shown above, we have $\sigma(M^{-1}L)\subseteq \mathbb{R_{++}}$. 
	Secondly, since $L$ is nonsingular, the eigenvalue $\lambda$ in \eqref{eq: observability in lossless general} cannot be zero.
	%
    %
	Based on the foregoing two observations, when the pair $(M^{-1}L,M^{-1}D)$ is not observable, there must exist $\lambda \in \mathbb{R_{+}}, \lambda\ne0$ and $x \in \mathbb{C}^n , x\ne 0$ such that \eqref{eq: observability in lossless general} holds.
	Define $\xi=\sqrt{-\lambda}$, which is a purely imaginary number. The quadratic pencil $M^{-1}P(\xi)=\xi^2 I + \xi M^{-1}D + M^{-1}L$ is singular because $M^{-1}P(\xi)x = \xi^2 x + \xi M^{-1}Dx + M^{-1}Lx = -\lambda x + 0 + \lambda x = 0$. By  \cref{lemma: relation between ev J and ev J11}, $\xi$ is an eigenvalue of $J$. Similarly, we can show $-\xi$ is an eigenvalue of $J$. Therefore, the equilibrium point is not hyperbolic, which is a desired contradiction.
\end{proof}

As mentioned above, if matrix $D$ is nonsingular, the pair $(M^{-1}\nabla f(x_0),\allowbreak  M^{-1}D)$ is always observable. Indeed, if we replace the assumption $D\in\S^n_+$ with $D\in\S^n_{++}$ in \cref{thrm: nec and suf for pure imaginary lossless}, then the equilibrium point $(x_0,0)$ is not only hyperbolic but also asymptotically stable. This is proved in \Cref{thrm: Stability of Symmetric Second-Order Systems with nonsing damping} in \Cref{appendix: Stability of Symmetric Second-Order Systems with Nonsingular Damping}.

Another interesting observation is that when an equilibrium point is hyperbolic, \cref{thrm: nec and suf for pure imaginary lossless} confirms the monotonic behaviour of damping in \cref{thrm: monotonic damping behaviour}. Specifically, suppose an equilibrium point $(x_0,0)$ is hyperbolic for a value of damping matrix $D_I\in\S^n_+$. \cref{thrm: nec and suf for pure imaginary lossless} implies that the pair $(M^{-1}\nabla f(x_0),M^{-1}D_I)$ is observable. Note that if we change the damping to $D_{II}\in\S^n_+$ such that $D_{II}\succeq D_I$, then the pair $(M^{-1}\nabla f(x_0),M^{-1}D_{II})$ is also observable. Hence, the equilibrium point $(x_0,0)$ of the new system with damping $D_{II}$ is also hyperbolic. This is consistent with the monotonic behaviour of damping which is proved in \cref{thrm: monotonic damping behaviour}.

Under additional assumptions, \cref{thrm: nec and suf for pure imaginary lossless} can be partially generalized to a sufficient condition for the breakdown of hyperbolicity when $L$, $M$, and $D$ are not symmetric as in the following
\begin{theorem}[hyperbolicity in second-order systems: unsymmetric case]\label{thrm: nec and suf for pure imaginary lossy}
Consider the second-order ODE system \eqref{eq: nonlinear ode} with nonsingular inertia matrix $M\in\mathbb{R}^{n\times n}$ and damping matrix $D \in\mathbb{R}^{n\times n}$. Suppose $(x_0,0)\in \mathbb{R}^{n+n}$ is an equilibrium point of the corresponding first-order system \eqref{eq: nonlinear ode 1 order} with the Jacobian matrix $J\in\mathbb{R}^{2n\times 2n}$ defined in \eqref{eq: J general case} such that $L=\nabla f(x_0)\in \mathbb{R}^{n\times n}$. If $M^{-1}L$ has a positive eigenvalue $\lambda$ with eigenvector $x$ such that $x$ is in the nullspace of $M^{-1}D$, then the spectrum of the Jacobian matrix $\sigma(J)$ contains a pair of purely imaginary eigenvalues.
\end{theorem}
\begin{proof}
    The proof is similar to that of \Cref{thrm: nec and suf for pure imaginary lossless} and is given in \cref{proof of thrm: nec and suf for pure imaginary lossy}.
\end{proof}

\subsection{Bifurcation under Damping Variations}
\label{subsec: Bifurcation under Damping Variations}
In \Cref{subsec: Sufficient and Necessary Conditions}, we developed necessary and/or sufficient conditions for breaking the hyperbolicity through purely imaginary eigenvalues. Naturally, the next question is: what are the consequences of breaking the hyperbolicity? To answer this question, consider the parametric ODE 
\begin{align} \label{eq: parameter-dependent second order system}
     M \ddot{x} + D(\gamma) \dot{x} + f(x) = 0, 
\end{align}
which satisfies the same assumptions as \eqref{eq: nonlinear ode}. Suppose $D(\gamma)$ is a smooth function of $\gamma\in\mathbb{R}$, and $(x_0,0)\in \mathbb{R}^{n+n}$ is a hyperbolic equilibrium point of the corresponding first-order system at $\gamma=\gamma_1$ with the Jacobian matrix $J(x,\gamma)$ defined as
\begin{align}\label{eq: J general case hopd thrm proof}
	J(x, \gamma) = \begin{bmatrix}
	0 & I \\
	-M^{-1} \nabla f(x)  & - M^{-1} D(\gamma) \\
	\end{bmatrix} \in\mathbb{R}^{2n\times 2n}.
\end{align} 
Let us vary $\gamma$ from $\gamma_1$ to $\gamma_2$ and monitor the equilibrium point. There are two ways in which the hyperbolicity can be broken. Either a simple real eigenvalue approaches zero and we have $0\in\sigma(J(x_0,\gamma_2))$, or a pair of simple complex eigenvalues approaches the imaginary axis and we have $\pm \IU  \omega_0 \in \sigma( J(x_0,\gamma_2) )$ for some $\omega_0>0$. The former corresponds to a fold bifurcation, while the latter is associated with a Hopf (more accurately, Poincare-Andronov-Hopf) bifurcation\footnote{It can be proved that we need more parameters to create extra eigenvalues on the imaginary axis unless the system has special properties such as symmetry \cite{2004-kuznetsov-hopf}.}. The next theorem states the precise conditions for a Hopf bifurcation to occur in system \eqref{eq: parameter-dependent second order system}.

\begin{theorem} \label{coro: Hopf bifurcation}
	Consider the parametric ODE \eqref{eq: parameter-dependent second order system}, with inertia matrix $M\in\mathbb{S}^n_{++}$ and damping matrix $D(\gamma) \in\mathbb{S}^n_+$. Suppose $D(\gamma)$ is a smooth function of $\gamma$, $(x_0,0)\in \mathbb{R}^{n+n}$ is an isolated equilibrium point of the corresponding first-order system, and $L:=\nabla f(x_0)\in \mathbb{S}^n_{++}$. Assume the following conditions are satisfied:
	\begin{enumerate}[(i)]
		\item There exists $\gamma_0\in\mathbb{R}$ such that the pair $(M^{-1}L, M^{-1}D(\gamma_0) )$ is not observable, that is, $\exists \lambda\in\mathbb{C},v\in\mathbb{C}^n, v\ne0$ such that 
		\begin{align} \label{eq: hopf observ}
		M^{-1} L v = \lambda v \text{ and } M^{-1}D(\gamma_0)v=0.
		\end{align}\label{coro: hopf conditions_1}
		\vspace{-5mm}
		\item $\IU \omega_0$ is a simple eigenvalue of $J(x_0,\gamma_0)$, where $\omega_0=\sqrt{\lambda}$. 
		\label{coro: hopf conditions_3}
		\item $\mathrm{Im}( q^* M^{-1} D'(\gamma_0) v ) \ne 0$, where $D'(\gamma_0)$ is the derivative of $D(\gamma)$ at $\gamma=\gamma_0$, $\ell_0=(p,q)\in\mathbb{C}^{n+n}$ is a left eigenvector of $J(x_0, \gamma_0)$ corresponding to eigenvalue $\IU  \omega_0$, and $\ell_0$ is normalized so that $\ell_0^*r_0 = 1$ where $r_0=(v,\IU \omega_0 v)$.
		\label{coro: hopf conditions_transvers} 
		%
		%
		\item $\det({P(\kappa)})\ne0$ for all $\kappa\in\mathbb{Z} \setminus \{-1,1\}$, where $P(\kappa)$ is the quadratic matrix pencil given by $P(\kappa):= \nabla f(x_0)-\kappa^2\omega_0^2 M + \IU \kappa\omega_0 D(\gamma_0)$.
		\label{coro: hopf conditions_4}
        %
		%
	\end{enumerate}
	Then, there exists smooth functions $\gamma=\gamma(\varepsilon)$ and $T=T(\varepsilon)$ depending on a parameter $\varepsilon$ with $\gamma(0)=\gamma_0$ and $T(0)=2\pi |\omega_0|^{-1}$ such that there are nonconstant periodic solutions of \eqref{eq: parameter-dependent second order system} with period $T(\varepsilon)$ which collapses into the equilibrium point $(x_0,0)$ as $\varepsilon\to 0$.
\end{theorem}
\begin{proof}
	By \cref{thrm: nec and suf for pure imaginary lossless}, condition (\ref{coro: hopf conditions_1}) implies that the Jacobian matrix \eqref{eq: J general case hopd thrm proof} at $(x,\gamma) = (x_0,\gamma_0)$ possesses a pair of purely imaginary eigenvalues $\pm \IU  \omega_0$, where $\omega_0=\sqrt{\lambda}$. Moreover, a right eigenvector of $\IU  \omega_0$ is $(v, \IU  \omega_0 v)$, where $v$ is from \eqref{eq: hopf observ}. 
	According to condition (\ref{coro: hopf conditions_3}), the eigenvalue $\IU \omega_0$ is simple. Therefore, according to the eigenvalue perturbation theorem \cite[Theorem 1]{2020-greenbaum-perturbation}, for $\gamma$ in a neighborhood of $\gamma_0$, the matrix $J(x_0,\gamma)$ has an eigenvalue $\xi(\gamma)$ and corresponding right and left eigenvectors $r(\gamma)$ and $\ell(\gamma)$ with $\ell(\gamma)^*r(\gamma)=1$ such that $\xi(\gamma)$, $r(\gamma)$, and $\ell(\gamma)$ are all analytic functions of $\gamma$, satisfying $\xi(\gamma_0)=\IU \omega_0$, $r(\gamma_0)=r_0$, and $\ell(\gamma_0)=\ell_0$. Let us differentiate the equation $J(x_0,\gamma) r(\gamma) = \xi(\gamma) r(\gamma)$ and set $\gamma=\gamma_0$, to get
	\begin{align}
	    J'(x_0,\gamma_0) r(\gamma_0) + J(x_0,\gamma_0) r'(\gamma_0) = \xi'(\gamma_0) r(\gamma_0) + \xi(\gamma_0) r'(\gamma_0).
	\end{align}
	After left multiplication by $\ell_0^*$, and using $\ell_0^*r_0=1$, we obtain the derivative of $\xi(\gamma)$ at $\gamma = \gamma_0$:
	\begin{align*}
	\xi'(\gamma_0) =  
	\begin{bmatrix}
	 p^{*} & q^{*}
	\end{bmatrix}
	\begin{bmatrix}
	0 & 0 \\
	0  & - M^{-1} D'(\gamma_0) \\
	\end{bmatrix} 
	\begin{bmatrix}
	v \\ \IU \omega_0 v
	\end{bmatrix} = - \IU \omega_0 q^{*} M^{-1} D'(\gamma_0) v.
	\end{align*}
	%
	%
    %
    %
    %
	Now, $\mathrm{Im}( q^* M^{-1} D'(\gamma_0) v ) \ne 0$ in condition (\ref{coro: hopf conditions_transvers}) implies that $\mathrm{Re}(\xi'(\gamma_0)) \ne 0$ which is a necessary condition for Hopf bifurcation. 
    %
    %
	Therefore, the results follow from the Hopf bifurcation theorem \cite[Section 2]{1978-Schmidt-hopf}. Note that $J(x_0,\gamma)$ is singular if and only if $\nabla f(x_0)$ is singular. Thus, nonsingularity of $\nabla f(x_0)$ is necessary for Hopf bifurcation.
\end{proof}

If one or more of the listed conditions in \cref{coro: Hopf bifurcation} are not satisfied, we may still have the birth of a periodic orbit but some of the conclusions of the theorem may not hold true. The bifurcation is then called a degenerate Hopf bifurcation. For instance, if the transversality condition (\ref{coro: hopf conditions_transvers}) is not satisfied, the stability of the equilibrium point may not change, or multiple periodic orbits may bifurcate \cite{1978-Schmidt-hopf}. The next theorem describes a safe region for damping variations such that fold and Hopf bifurcations will be avoided.

\begin{theorem}\label{thm: fold and Hopf bifurcation}
 Consider the parametric ODE \eqref{eq: parameter-dependent second order system}, with a nonsingular inertia matrix $M\in\mathbb{R}^{n\times n}$. Suppose the damping matrix $D(\gamma)\in\mathbb{R}^{n\times n}$ is a smooth function of $\gamma$, $(x_0,0)\in \mathbb{R}^{n+n}$ is a hyperbolic equilibrium point of the corresponding first-order system at $\gamma=\gamma_0$, and $L:=\nabla f(x_0)\in\mathbb{R}^{n\times n}$. Then, the following statements hold:
 \begin{enumerate}[(i)]
     \item Variation of $\gamma$ in $\mathbb{R}$ will not lead to any fold bifurcation.
     %
     %
     \item Under the symmetric setting, i.e., when $M\in\mathbb{S}^n_{++}$, $D(\gamma) \in\mathbb{S}^n_+$, and $L\in \mathbb{S}^n_{++}$,  variation of $\gamma$ in $\mathbb{R}$ will not lead to any Hopf bifurcation, as long as $D(\gamma) \succeq D(\gamma_0)$. If in addition the equilibrium point is stable, variation of $\gamma$ in $\mathbb{R}$ will not make it unstable, as long as $D(\gamma) \succeq D(\gamma_0)$.  
     %
 \end{enumerate}
  %
\end{theorem}
\begin{proof}
    According to \cref{lemma: relation between ev J and ev J11}, zero is an eigenvalue of $J$ if and only if the matrix $L=\nabla f(x_0)$ is singular. Therefore, the damping matrix $D$ has no role in the zero eigenvalue of $J$. The second statement follows from \Cref{thrm: monotonic damping behaviour}.
\end{proof}
The above theorem can be straightforwardly generalized to bifurcations having higher codimension.

\section{Power System Models and Impact of Damping}
\label{Sec: Power System Model and Its Properties}
The Questions \eqref{Q1}-\eqref{Q3} asked in \Cref{subsec: Problem Statement} and the  theorems and results discussed in the previous parts of this paper arise naturally from the foundations of electric power systems. These results are useful tools for analyzing the behaviour and maintaining the stability of power systems. In the rest of this paper, we focus on power systems to further explore the role of damping in these systems.
\subsection{Power System Model} \label{subsec: Multi-Machine Swing Equations}
Consider a power system with the set of interconnected generators $\mathcal{N}=\{1,\cdots,n\}, n\in\mathbb{N}$. Based on the classical small-signal stability assumptions \cite{2008-anderson-stability}, the mathematical model for a power system is described by the following second-order system:
\begin{align} \label{eq: 2nd order swing}
		\frac{m_j}{\omega_s} \ddot{\delta}_j(t)+ \frac{d_j}{\omega_s} {\dot{\delta}}_j(t) = P_{m_j} - P_{e_j}(\delta(t)) && \forall j \in \mathcal{N}.
\end{align}
Considering the state space $\mathcal{S}:= \{    (\delta,\omega): \delta \in \mathbb{R}^n,  \omega \in \mathbb{R}^n \}$, the dynamical system \eqref{eq: 2nd order swing} can be represented as a system of first-order nonlinear autonomous ODEs, aka swing equations:  
\begin{subequations} \label{eq: swing equations}
	\begin{align}
		& \dot{\delta}_j(t) = \omega_j(t) && \forall j \in \mathcal{N},  \label{eq: swing equations a}\\
		&  \frac{m_j}{\omega_s} \dot{\omega}_j(t)+ \frac{d_j}{\omega_s} \omega_j(t) = P_{m_j} - P_{e_j}(\delta(t)) && \forall j \in \mathcal{N}, \label{eq: swing equations b}
	\end{align}
\end{subequations}
where for each generator $j\in\mathcal{N}$, $P_{m_j}$ and $P_{e_j}$ are mechanical and electrical powers in per unit, $m_j$ is the inertia constant in seconds, $d_j$ is the unitless damping coefficient, $\omega_{s}$ is the synchronous angular velocity in electrical radians per seconds, $t$ is the time in seconds, $\delta_j(t)$ is the rotor electrical angle in radians, and finally $\omega_j(t)$ is the deviation of the rotor angular velocity from the synchronous velocity in electrical radians per seconds. For the sake of simplicity, henceforth we do not explicitly write the dependence of the state variables $\delta$ and $\omega$ on time $t$. 
The electrical power $P_{e_j}$ in \eqref{eq: swing equations b} can be further spelled out: 
\begin{align} \label{eq: flow function}
	P_{e_j}(\delta) & = \sum \limits_{k = 1}^n { V_j  V_k Y_{jk} \cos \left( \theta _{jk} - \delta _j + \delta _k \right)}
\end{align}
where $V_j$ is the terminal voltage magnitude of generator $j$, and
$Y_{jk}\exp{({\IU \theta_{jk}})}$ is the $(j,k)$ entry of the reduced admittance matrix, with $Y_{jk}\in\mathbb{R}$ and $\theta_{jk}\in\mathbb{R}$. The reduced admittance matrix encodes the underlying graph structure of the power grid, which is assumed to be a connected graph in this paper.
Note that for each generator $j\in\mathcal{N}$, the electrical power $P_{e_j}$ in general is a function of angle variables $\delta_k$ for all $k\in\mathcal{N}$. Therefore, the dynamics of generators are interconnected through the function $P_{e_j}(\delta)$ in \eqref{eq: 2nd order swing} and \eqref{eq: swing equations}.
\begin{definition}[flow function] \label{def: flow function}
	The smooth function $P_e:\mathbb{R}^n \to \mathbb{R}^n$ given by $\delta \mapsto P_e(\delta)$ in \eqref{eq: flow function} is called the flow function.
\end{definition}
The smoothness of the flow function (it is $\mathcal{C}^\infty$ indeed) is a sufficient condition for the existence and uniqueness of the solution to the ODE \eqref{eq: swing equations}. 
The flow function is translationally invariant with respect to the operator $\delta \mapsto \delta + \alpha \mathbf{1}$, where $\alpha \in \mathbb{R}$ and $\mathbf{1}\in\mathbb{R}^n$ is the vector of all ones. In other words, $P_e(\delta + \alpha \mathbf{1})=P_e(\delta)$. A common way to deal with this situation is to define a reference bus and refer all other bus angles to it. This is equivalent to projecting the original state space $\mathcal{S}$ onto a lower dimensional space. We will delve into this issue in \Cref{subsec: Referenced Model}.
%
\subsection{Jacobian of Swing Equations} \label{SubSec: Graph Theory Interpretation and Spectral Properties}
Let us take the state variable vector $(\delta ,\omega)\in\mathbb{R}^{2n}$ into account and note that the Jacobian of the vector field in \eqref{eq: swing equations} has the form \eqref{eq: J general case}
where $M= \frac{1}{\omega_s}  \mathbf{diag}(m_1,\cdots,m_n)$ and $D=\frac{1}{\omega_s}\mathbf{diag}(d_1,\cdots,d_n)$. Moreover, $f=P_e-P_m$ and $\nabla f =  \nabla P_e (\delta) \in\mathbb{R}^{n\times n}$ is the Jacobian of the flow function with the entries:
\begin{align*} 
	&  \frac {\partial P_{e_j}} {\partial \delta_j} = \sum \limits_{k \ne j} { V_j V_k Y_{jk} \sin \left( {\theta _{jk} - {\delta _j} + {\delta _k}} \right) }, \forall j \in \mathcal{N}\\
	&  \frac{\partial P_{e_j} } {\partial \delta _k} =  - {V_j} {V_k}{Y_{jk}}\sin \left( {{\theta _{jk}} - {\delta _j} + {\delta _k}} \right),\forall j,k \in \mathcal{N}, k \neq j.
\end{align*}
%
Let $\mathcal{L}$ be the set of transmission lines of the reduced power system. We can rewrite $ {\partial P_{e_j}}/ {\partial \delta_j}=\sum_{k=1, k\ne j}^n w_{jk}$ and ${\partial P_{e_j} }/ {\partial \delta _k}=-w_{jk}$, where 
\begin{align} \label{eq: weights of digraph lossy}
w_{jk} = 
\begin{cases}
 {V_j} {V_k}{Y_{jk}}\sin \left( \varphi_{jk} \right)  & \forall \{ j,k \} \in \mathcal{L} \\
 0 & \text{otherwise},
\end{cases}
\end{align}
and $\varphi_{jk} = {{\theta _{jk}} - {\delta _j} + {\delta _k}}$. Typically, 
we have $\varphi_{jk} \in ( 0,\pi )$ for all $\{j,k \} \in \mathcal{L}$ \cite{2020-fast-certificate}. Thus, it is reasonable to assume that the equilibrium points $(\delta^{0},\omega^{0})$ of the dynamical system \eqref{eq: swing equations} are located in the set $\Omega$ defined as
\begin{align*}
\Omega = \left\{ (\delta,\omega)\in\mathbb{R}^{2n} :   0 < \theta_{jk}-\delta_j+\delta_k < \pi , \forall \{j,k\} \in \mathcal{L},\omega = 0  \right\}.
\end{align*}
Under this assumption, the terms $w_{jk} > 0$ for all transmission lines $\{j,k\} \in \mathcal{L}$.
Consequently, $ {\partial P_{e_j}} / {\partial \delta_j}\ge0, \forall j \in \mathcal{N}$ and ${\partial P_{e_j} }/ {\partial \delta _k} \le 0, \forall j,k \in \mathcal{N}, k \neq j$. Moreover, $\nabla P_e (\delta)$ has a zero row sum, i.e., $ \nabla P_e (\delta)\mathbf{1} = 0 \implies 0 \in \sigma(\nabla P_e (\delta))$. Given these properties, $\nabla P_e (\delta^0)$ turns out to be a singular M-matrix for all $(\delta^{0},\omega^{0})\in\Omega$ \cite{2020-fast-certificate}. Recall that a matrix $A$ is an \emph{M-matrix} if the off-diagonal elements of $A$ are nonpositive and the nonzero eigenvalues of $A$ have positive real parts \cite{1974-M-matrices}. Finally, if the power system under study has a connected underlying undirected graph, the zero eigenvalue of $\nabla P_e (\delta^0)$ will be simple \cite{2020-fast-certificate}. 
%

    In general, the Jacobian $\nabla P_e (\delta)$ is not symmetric. When the power system is \emph{lossless}, i.e., when the transfer conductances of the grid are zero, then $\theta_{jj}=-\frac{\pi}{2}, \forall j \in \mathcal{N}$ and $\theta_{jk}=\frac{\pi}{2}, \forall \{j,k\} \in \mathcal{L}$. In a lossless system, $\nabla P_e (\delta)$ is symmetric. If in addition an equilibrium point $(\delta^0,\omega^0)$ belongs to the set $\Omega$, then $\nabla P_e (\delta^0) \in\mathbb{S}^n_+$, because $\nabla P_e (\delta^0)$ is real symmetric and diagonally dominant \cite{2021-gholami-sun-MMG-stability}. 

\subsection{Referenced Power System Model}
\label{subsec: Referenced Model}

The translational invariance of the flow function $P_e$ gives rise to a zero eigenvalue in the spectrum of $\nabla P_e (\delta)$, and as a consequence, in the spectrum of $J(\delta)$. This zero eigenvalue and the corresponding nullspace pose some difficulties in monitoring the hyperbolicity of the equilibrium points, specially during Hopf bifurcation analysis. As mentioned in \Cref{subsec: Multi-Machine Swing Equations}, this situation can be dealt with by defining a reference bus and referring all other bus angles to it. Although this is a common practice in power system context \cite{2008-anderson-stability}, the spectral and dynamical relationships between the original system and the referenced system are not rigorously analyzed in the literature. In this section, we fill this gap to facilitate our analysis in the later parts.

\subsubsection{Referenced Model}
Define $\psi_j:=\delta_j-\delta_n, \forall j \in \{1,2,...,n-1\}$ and reformulate the swing equation model \eqref{eq: swing equations} into the \emph{referenced model}
\begin{subequations} \label{eq: Swing Equation Polar referenced}
	\begin{align}
	& \dot{\psi}_j = \omega_j - \omega_n \: \qquad \qquad \qquad \qquad \forall j \in \{1,...,n-1\},  \\
	& \dot{\omega}_j = - \frac{d_j}{m_j} \omega_j + \frac{\omega_s}{m_j}  (P_{m_j} - P_{e_j}^r(\psi)) \: \: \: \: \forall j \in \{1,...,n\},
	\end{align}
\end{subequations}
where for all $j$ in $\{1,...,n\}$ we have
	\begin{align} \label{eq: flow function referenced compact}
	P_{e_j}^r(\psi) & = \sum \limits_{k = 1}^{n} { V_j  V_k  Y_{jk} \cos \left( \theta _{jk} - \psi_j + \psi_k \right)},
	\end{align}
and $\psi_n\equiv0$. The function $P_e^r: \mathbb{R}^{n-1} \to \mathbb{R}^n$ given by \eqref{eq: flow function referenced compact} is called the referenced flow function. 
\subsubsection{The Relationship}
We would like to compare the behaviour of the two dynamical systems \eqref{eq: swing equations} and \eqref{eq: Swing Equation Polar referenced}. Let us define the linear mapping $\Psi:\mathbb{R}^n\times\mathbb{R}^n \to \mathbb{R}^{n-1} \times\mathbb{R}^n$ given by $(\delta,\omega)\mapsto(\psi,\omega)$ such that
\begin{align*}
   \Psi(\delta,\omega) = \left  \{ (\psi,\omega) : \psi_j:=\delta_j-\delta_n, \forall j \in \{1,2,...,n-1\}  \right  \}. 
\end{align*}
This map is obviously smooth but not injective.
It can also be written in matrix form 
\begin{align*}
       \Psi(\delta,\omega) = \begin{bmatrix} T_1 & 0 \\ 0 & I_n \end{bmatrix}
       \begin{bmatrix} \delta \\ \omega \end{bmatrix}
\end{align*}
where $I_n\in\mathbb{R}^{n\times n}$ is the identity matrix, $\mathbf{1}\in\mathbb{R}^{n-1}$ is the vector of all ones, and
\begin{align}
T_1:=\begin{bmatrix}
I_{n-1}& -\mathbf{1}
\end{bmatrix} \in \mathbb{R}^{(n-1)\times n}.
\end{align}


The next proposition, which is proved in \cref{proof of thrm: original model vs referenced model}, establishes the relationship between the original model \eqref{eq: swing equations} and the referenced model \eqref{eq: Swing Equation Polar referenced}.
\begin{proposition} \label{thrm: original model vs referenced model}
Let $(\delta^0,\omega^0)$ be an equilibrium point of the swing equation \eqref{eq: swing equations} and $(n_-,n_0,n_+)$ be the inertia\footnote{Inertia of a matrix (see e.g. \cite{2013-Horn-matrix-analysis} for a definition) should not be confused with the inertia matrix $M$.} of its Jacobian at this equilibrium point. The following two statements hold:
\begin{enumerate}[(i)]
    \item $\Psi(\delta^0,\omega^0)$ is an equilibrium point of the referenced model \eqref{eq: Swing Equation Polar referenced}.
    \item $(n_-,n_0-1,n_+)$ is the inertia of the Jacobian of \eqref{eq: Swing Equation Polar referenced} at $\Psi(\delta^0,\omega^0)$.
\end{enumerate}
\end{proposition}

\begin{remark}
    Note that the equilibrium points of the referenced model \eqref{eq: Swing Equation Polar referenced} are in the set 
    \begin{align*}
        \Tilde{\mathcal{E}} =  \biggl\{(\psi,\omega)\in \mathbb{R}^{n-1}\times\mathbb{R}^n : \; & \omega_j = \omega_n \:, \forall j \in \{1,...,n-1\}, \\ & P_{m_j} = P_{e_j}^r(\psi) + d_j \omega_n/\omega_s , \: \forall j \in \{1,...,n\}   \biggr\}
    \end{align*}
    where $\omega_n$ is not necessarily zero. Therefore, the referenced model \eqref{eq: Swing Equation Polar referenced} may have extra equilibrium points which do not correspond to any equilibrium point of the original model \eqref{eq: swing equations}.
 \end{remark}

\subsection{Impact of Damping in Power Systems}
\label{Sec: Impact of Damping in Power Systems}
The theoretical results in \Cref{Sec: Monotonic Behavior of Damping} and \Cref{Sec: Impact of Damping in Hopf Bifurcation} have important applications in electric power systems. For example, \cref{thrm: monotonic damping behaviour} is directly applicable to lossless power systems, and provides new insights to improve the situational awareness of power system operators. Recall that many control actions taken by power system operators are directly or indirectly targeted at changing the effective damping of the system \cite{1994-kundur-stability,2019-Patrick-koorehdavoudi-input,2012-aminifar-wide-area-damping}. In this context, \cref{thrm: monotonic damping behaviour} determines how the system operator should change the damping of the system in order to avoid breaking the hyperbolicity and escaping dangerous bifurcations. 
    Now, consider the case where a subset of power system generators have zero damping coefficients. Such partial damping is possible in practice specially in inverter-based resources (as damping coefficient corresponds to a controller parameter which can take zero value). The next theorem and remark follow from \cref{thrm: nec and suf for pure imaginary lossless}, and show how partial damping could break the hyperbolicity in lossless power systems.
    %
    %

\begin{theorem}[purely imaginary eigenvalues in lossless power systems] \label{prop: nec and suf for pure imaginary lossless power system}
    Consider a lossless network \eqref{eq: swing equations} with an equilibrium point $(\delta^{0},\omega^{0})\in \Omega$. Suppose all generators have positive inertia and nonnegative damping coefficients. Then, $\sigma(J(\delta^{0}))$ contains a pair of purely imaginary eigenvalues if and only if the pair $(M^{-1}\nabla P_e (\delta^0), \allowbreak M^{-1}D)$ is not observable.
\end{theorem}

\begin{proof}
    As mentioned above, we always assume the physical network connecting the power generators is a connected (undirected) graph. Under this assumption,
    as mentioned in \Cref{SubSec: Graph Theory Interpretation and Spectral Properties}, matrix $L:=\nabla P_e (\delta^0)$ has a simple zero eigenvalue with a right eigenvector $\mathbf{1}\in\mathbb{R}^{n}$, which is the vector of all ones \cite{2020-fast-certificate}.
    Moreover, since the power system is lossless and $(\delta^{0},\omega^{0})\in \Omega$, we have $L\in\mathbb{S}^n_+$.
    If $D=0$, the pair $(M^{-1}L,M^{-1}D)$ can never be observable. Using a similar argument as in the first part in the proof of \cref{thrm: nec and suf for pure imaginary lossless}, it can be shown that $M^{-1}L$ has a simple zero eigenvalue and the rest of its eigenvalues are positive, i.e., $\sigma(M^{-1}L)\subseteq \mathbb{R_{+}}=\{\lambda\in\mathbb{R}:\lambda\ge0\}$. Meanwhile, when $D=0$, we have $\mu\in\sigma(J)\iff\mu^2\in\sigma(-M^{-1}L)$. Notice that a power grid has at least two nodes, i.e. $n\ge2$, and hence, $M^{-1}L$ has at least one positive eigenvalue, i.e., $\exists \lambda\in\R_+, \lambda>0$ such that $\lambda\in\sigma(M^{-1}L)$. Hence, $\mu= \sqrt{-\lambda}$ is a purely imaginary number and is an eigenvalue of $J$. Similarly, we can show that $-\mu$ is an eigenvalue of $J$. Consequently, the theorem holds in the case of $D=0$. In the sequel, we assume that $D\not=0$.
	%

   
	\emph{Necessity:}
	Assume there exists $\beta>0$ such that $\IU \beta \in \sigma(J)$. We will show that the pair $(M^{-1}L,M^{-1}D)$ is not observable.
	%
	By \cref{lemma: relation between ev J and ev J11}, $\IU\beta \in \sigma(J)$ if and only if the matrix pencil $(M^{-1}L + \IU\beta M^{-1}D - \beta^2 I)$ is singular: 
	\begin{align*} 
	& \det \left(  M^{-\frac{1}{2}}(M^{-\frac{1}{2}}LM^{-\frac{1}{2}} + \IU \beta  M^{-\frac{1}{2}}DM^{-\frac{1}{2}} - \beta^2 I) M^{\frac{1}{2}} \right) = 0,
	\end{align*}
	or equivalently, $\exists (x + \IU y) \ne 0$ such that $x,y\in\mathbb{R}^n$ and
	\begin{align} \label{eq: proof of parially damped hopf bifurcation lossless swing}
	\notag & (M^{-\frac{1}{2}}LM^{-\frac{1}{2}} + \IU \beta  M^{-\frac{1}{2}} D M^{-\frac{1}{2}} - \beta^2 I)(x+\IU y) = 0 \\
	& \iff  \begin{cases}
	(M^{-\frac{1}{2}}LM^{-\frac{1}{2}}-\beta^2 I)x - \beta M^{-\frac{1}{2}}DM^{-\frac{1}{2}} y = 0, \\ 
	(M^{-\frac{1}{2}}LM^{-\frac{1}{2}}-\beta^2 I)y + \beta M^{-\frac{1}{2}}DM^{-\frac{1}{2}} x = 0.
	\end{cases}
	\end{align}
	Let $\hat{L}:=M^{-\frac{1}{2}}LM^{-\frac{1}{2}}$, $\hat{D}:=M^{-\frac{1}{2}}DM^{-\frac{1}{2}}$, and observe that
	\begin{align*}
	\begin{cases}
	y^\top (\hat{L}-\beta^2 I)x = y^\top (\beta \hat{D} y) = \beta y^\top \hat{D} y  \ge 0,
	\\
	x^\top (\hat{L}-\beta^2 I)y = x^\top (-\beta \hat{D} x) = -\beta x^\top \hat{D} x  \le 0,
	\end{cases}
	\end{align*}
    %
	where we have used the fact that $\hat{D}$ is *-congruent to $D$. According to Sylvester’s law of inertia, $\hat{D}$ and $D$ have the same inertia. Since $D\succeq 0$, we conclude that $\hat{D}\succeq0$.
	As $(\hat{L}-\beta^2 I)$ is symmetric, we have $x^\top (\hat{L}-\beta^2 I)y = y^\top (\hat{L}-\beta^2 I)x$. Therefore, we must have $x^\top \hat{D} x = y^\top \hat{D} y  =0$. Since $\hat{D}\succeq0$, we can infer that $x\in\ker(\hat{D})$ and $y\in\ker(\hat{D})$.
	%
	Now considering $\hat{D} y = 0$ and using the first equation in \eqref{eq: proof of parially damped hopf bifurcation lossless swing} we get 
	\begin{align}
	( \hat{L}-\beta^2 I)x = 0 \iff M^{-\frac{1}{2}}LM^{-\frac{1}{2}} x=\beta^2 x,
	\end{align}
	multiplying both sides from left by $M^{-\frac{1}{2}}$ we get $M^{-1}L (M^{-\frac{1}{2}} x) = \beta^2 (M^{-\frac{1}{2}} x)$. Thus, $ \hat{x}:= M^{-\frac{1}{2}} x$ is an eigenvector of $M^{-1}L$. Moreover, we have 
	\begin{align*}
	    M^{-1}D \hat{x}  = M^{-1}DM^{-\frac{1}{2}} x = M^{-\frac{1}{2}} (  \hat{D}  x) = 0,    
	\end{align*}
	which means that the pair $(M^{-1}L,M^{-1}D)$ is not observable.

	
	%
	
	\emph{Sufficiency:}
	Suppose the pair~ $\allowbreak (M^{-1}L, M^{-1}D)$ is not observable. We will show that $\sigma(J)$ contains a pair of purely imaginary eigenvalues.
	According to \cref{def: Observability}, $\exists \lambda \in \mathbb{C}, x \in \mathbb{C}^n , x\ne 0$ such that
	\begin{align} \label{eq: observability in lossless general swing}
	 M^{-1}Lx = \lambda x \text{ and } M^{-1}Dx = 0.
	\end{align}
	We make the following two observations.
	Firstly, as it is shown above, we have $\sigma(M^{-1}L)\subseteq \mathbb{R_{+}}$. 
	Secondly, 
	%
    $L$ has a simple zero eigenvalue and a one-dimensional nullspace spanned by $\mathbf{1}\in\R^n$. We want to emphasize that this zero eigenvalue of $L$ cannot break the observability of the pair $(M^{-1}L,M^{-1}D)$. Note that $\ker(L)=\ker(M^{-1}L)$ and $M^{-1}L \mathbf{1} =0$ implies that $M^{-1}D \mathbf{1} \ne0$ because $D\ne0$.
    %
	Based on the foregoing two observations, when the pair $(M^{-1}L,M^{-1}D)$ is not observable, there must exist $\lambda \in \mathbb{R_{+}}, \lambda\ne0$ and $x \in \mathbb{C}^n , x\ne 0$ such that \eqref{eq: observability in lossless general swing} holds.
	Define $\xi=\sqrt{-\lambda}$, which is a purely imaginary number. The quadratic pencil $M^{-1}P(\xi)=\xi^2 I + \xi M^{-1}D + M^{-1}L$ is singular because $M^{-1}P(\xi)x = \xi^2 x + \xi M^{-1}Dx + M^{-1}Lx = -\lambda x + 0 + \lambda x = 0$. By  \cref{lemma: relation between ev J and ev J11}, $\xi$ is an eigenvalue of $J$. Similarly, we can show $-\xi$ is an eigenvalue of $J$. Therefore, $\sigma(J)$ contains the pair of purely imaginary eigenvalues $\pm\xi$.
\end{proof}

The following remark illustrates how \cref{prop: nec and suf for pure imaginary lossless power system} can be used in practice to detect and damp oscillations in power systems.
\begin{remark}
	Consider the assumptions of \cref{prop: nec and suf for pure imaginary lossless power system} and suppose there exists a pair of purely imaginary eigenvalues $\pm \IU  \beta \in \sigma(J(\delta^0))$ which give rise to Hopf bifurcation and oscillatory behaviour of the system. This issue can be detected by observing the osculations in power system state variables (through phasor measurement units (PMUs) \cite{2012-aminifar-wide-area-damping}). According to \cref{prop: nec and suf for pure imaginary lossless power system}, we conclude that $\beta^2\in\sigma(M^{-1}\nabla P_e (\delta^0))$. Let $\mathcal{X}:= \{x^1,...,x^k\}$ be a set of independent eigenvectors associated with the eigenvalue $\beta^2\in\sigma(M^{-1}\nabla P_e (\delta^0))$, i.e., we assume that the corresponding eigenspace is $k$-dimensional. According to \cref{prop: nec and suf for pure imaginary lossless power system}, we have $M^{-1}Dx^\ell=0,\forall x^\ell\in\mathcal{X}$, or equivalently, $Dx^\ell=0,\forall x^\ell\in\mathcal{X}$. Since $D$ is diagonal, we have $d_jx_j^\ell=0,\forall j\in\{1,\cdots,n\}, \forall x^\ell\in\mathcal{X}$.
	In order to remove the purely imaginary eigenvalues, we need to make sure that $\forall x^\ell\in\mathcal{X}, \exists j \in \{1,\cdots,n\}$ such that $d_j x_j^\ell\ne0$. This can be done for each $x^\ell\in\mathcal{X}$ by choosing a $j\in\{1,\cdots,n\}$ such that $x_j^\ell\ne0$ and then increase the corresponding damping $d_j$ from zero to some positive number, thereby rendering the pair $(M^{-1}\nabla P_e (\delta^0),M^{-1}D)$ observable.
\end{remark}

\Cref{prop: nec and suf for pure imaginary lossless power system} gives a necessary and sufficient condition for the existence of purely imaginary eigenvalues in a lossless power system with \emph{nonnegative} damping and positive inertia. It is instructive to compare it with an earlier result in \cite{2021-gholami-sun-MMG-stability}, which shows that when all the generators in a lossless power system have \emph{positive} damping $d_j$ and positive inertia $m_j$, then any equilibrium point in the set $\Omega$ is asymptotically stable.
This is also proved in \Cref{thrm: Stability of Symmetric Second-Order Systems with nonsing damping} of \Cref{appendix: Stability of Symmetric Second-Order Systems with Nonsingular Damping} for the general second-order model \eqref{eq: nonlinear ode}.

Recall that the simple zero eigenvalue of the Jacobian matrix $J(\delta^0)$ in model \eqref{eq: swing equations} stems from the translational invariance of
the flow function defined in \Cref{def: flow function}. As mentioned earlier, we can eliminate this eigenvalue by choosing a reference bus and refer all other bus angles to it. According to \Cref{thrm: original model vs referenced model}, aside from the simple zero eigenvalue, the Jacobians of the original model \eqref{eq: swing equations} and the referenced model \eqref{eq: Swing Equation Polar referenced} have the same number of eigenvalues with zero real part. Hence, \cref{prop: nec and suf for pure imaginary lossless power system} provides a necessary and sufficient condition for breaking the hyperbolicity in the referenced lossless power system model \eqref{eq: Swing Equation Polar referenced}.

%
%
%
%
%

In lossy power systems, matrix $\nabla P_e (\delta)$ may not be symmetric. In this case, \cref{thrm: nec and suf for pure imaginary lossy} can be used for detecting purely imaginary eigenvalues. Meanwhile, let us discuss some noteworthy cases in more detail. \Cref{prop:hyperbolicity n2 n3} asserts that in small lossy power networks  with only one undamped generator, the equilibrium points are always hyperbolic. The proof is provided in \cref{proof of prop:hyperbolicity n2 n3}.
\begin{theorem} \label{prop:hyperbolicity n2 n3}
Let $n\in\{2,3\}$ and consider an $n$-generator system with only one undamped generator. Suppose $(\delta^{0},\omega^{0})\in \Omega$ holds, the underlying undirected power network graph is connected, and $\nabla P_e (\delta)$. Then the Jacobian matrix $J(\delta^0)$ has no purely imaginary eigenvalues.
We allow the network to be lossy, but we assume ${\partial P_{e_j} }/ {\partial \delta _k}=0$ if and only if ${\partial P_{e_k} }/ {\partial \delta _j}=0$. The lossless case is a special case of this.
\end{theorem}

The following counterexample shows that as long as there are two undamped generators, the Jacobian $J(\delta)$ at an equilibrium point may have purely imaginary eigenvalues.
\begin{proposition} \label{prop: non-hyper example}
For any $n\ge 2$, consider an $(n+1)$-generator system with $2$ undamped generators and the following $(n+1)$-by-$(n+1)$ matrices $L=\nabla P_e (\delta^0)$, $D$, and $M$:
\begin{align*}
L = \begin{bmatrix} 
1 & -\frac{1}{n} & -\frac{1}{n} & \cdots & -\frac{1}{n} \\
-\frac{1}{n} & 1 & -\frac{1}{n} & \cdots & -\frac{1}{n} \\
\vdots & \vdots  & \vdots & \ddots & \vdots \\
-\frac{1}{n} & -\frac{1}{n} & -\frac{1}{n} & \cdots & 1
\end{bmatrix},\\
D= \mathbf{diag} ([0,0,d_3,d_4,\cdots,d_{n+1}]), \: M =  I_{n+1}.
\end{align*}
Then $\pm \IU \beta \in \sigma(J(\delta^0))$, where $\beta^2 = 1+\frac{1}{n}$.
\end{proposition}
\begin{proof}
Let $\beta^2 = 1+\frac{1}{n}$ and observe that $\rank(L-\beta^2 M)=1$ and $\rank(\beta D)=(n+1)-2=n-1$. The rank-sum inequality \cite{2013-Horn-matrix-analysis} implies that
\begin{align*}
\rank(L-\beta^2 M- \IU \beta D) &\le \rank(L-\beta^2 M)+\rank(- \IU\beta D) = 1+(n-1)=n,
\end{align*}
that is $\det\left( L + \IU  \beta D - \beta^2 M \right)  = 0$. Now according to \cref{lemma: relation between ev J and ev J11}, the latter is equivalent to $\IU  \beta \in \sigma(J(\delta^0))$. This completes the proof.
Also note that the constructed $L$ is not totally unrealistic for a power system.
\end{proof}

\section{Illustrative Numerical Power System Examples} \label{Sec: Computational Experiments}
Two case studies will be presented to illustrate breaking the hyperbolicity and the occurrence of Hopf bifurcation under damping variations. Additionally, we adopt the center manifold theorem to determine the stability of bifurcated orbits. Note that using the center manifold theorem, a Hopf bifurcation in an $n$-generator network essentially reduces to a planar system provided that aside from the two purely imaginary eigenvalues no other eigenvalues have zero real part at the bifurcation point. Therefore, for the sake of better illustration we focus on small-scale networks.
\subsection{Case $1$} \label{subsubsec: case 1}
Consider a $3$-generator system with $D=\mathbf{diag} ([\gamma,\gamma,1.5])$, $M=I_3$, $Y_{12}=Y_{13}=2Y_{23}=\IU $ p.u., $P_{m_1}=-\sqrt{3}$ p.u., and $P_{m_2}=P_{m_3}=\sqrt{3}/2$ p.u. The load flow problem for this system has the solution $V_j=1$ p.u. $\forall j$ and $\delta_1=0$, $\delta_2=\delta_3=\pi/3$. Observe that when $\gamma=0$, the pair $(M^{-1}\nabla P_e (\delta^0),M^{-1}D)$ is not observable, and \cref{prop: nec and suf for pure imaginary lossless power system} implies that the spectrum of the Jacobian matrix $\sigma(J)$ contains a pair of purely imaginary eigenvalues. Moreover, this system satisfies the assumptions of \cref{prop: non-hyper example}, and consequently, we have $\pm \IU  \sqrt{1.5} \in \sigma(J)$. In order to eliminate the zero eigenvalue (to be able to use the Hopf bifurcation theorem), we adopt the associated referenced model using the procedure described in \Cref{subsec: Referenced Model}. 
The conditions (\ref{coro: hopf conditions_1})-(\ref{coro: hopf conditions_4}) of \cref{coro: Hopf bifurcation} are satisfied (specifically, the transversality condition (\ref{coro: hopf conditions_transvers}) holds because $\mathrm{Im}( q^* M^{-1} D'(\gamma_0) v ) = -0.5$), and accordingly, a periodic orbit bifurcates at this point. 
To determine the stability of bifurcated orbit, we compute the \emph{first Lyapunov coefficient} $l_1(0)$ as described in \cite{2004-kuznetsov-hopf}. If the first Lyapunov coefficient is negative, the bifurcating limit cycle is stable, and the bifurcation is supercritical. Otherwise it is unstable and the bifurcation is subcritical. In this example, we get $l_1(0)=-1.7\times10^{-3}$ confirming that the type of Hopf bifurcation is supercritical and a stable limit cycle is born. Figs. (\ref{fig:sfig1 case1 bifurcation})-(\ref{fig:sfig3 case1 bifurcation}) depict these limit cycles when the parameter $\gamma$ changes. Moreover, Fig. (\ref{fig:sfig4 case1 bifurcation}) shows the oscillations observed in the voltage angles and frequencies when $\gamma=0$.
\begin{figure}
\centering
\begin{subfigure}{0.33\textwidth}
  \centering
  \includegraphics[width=\linewidth]{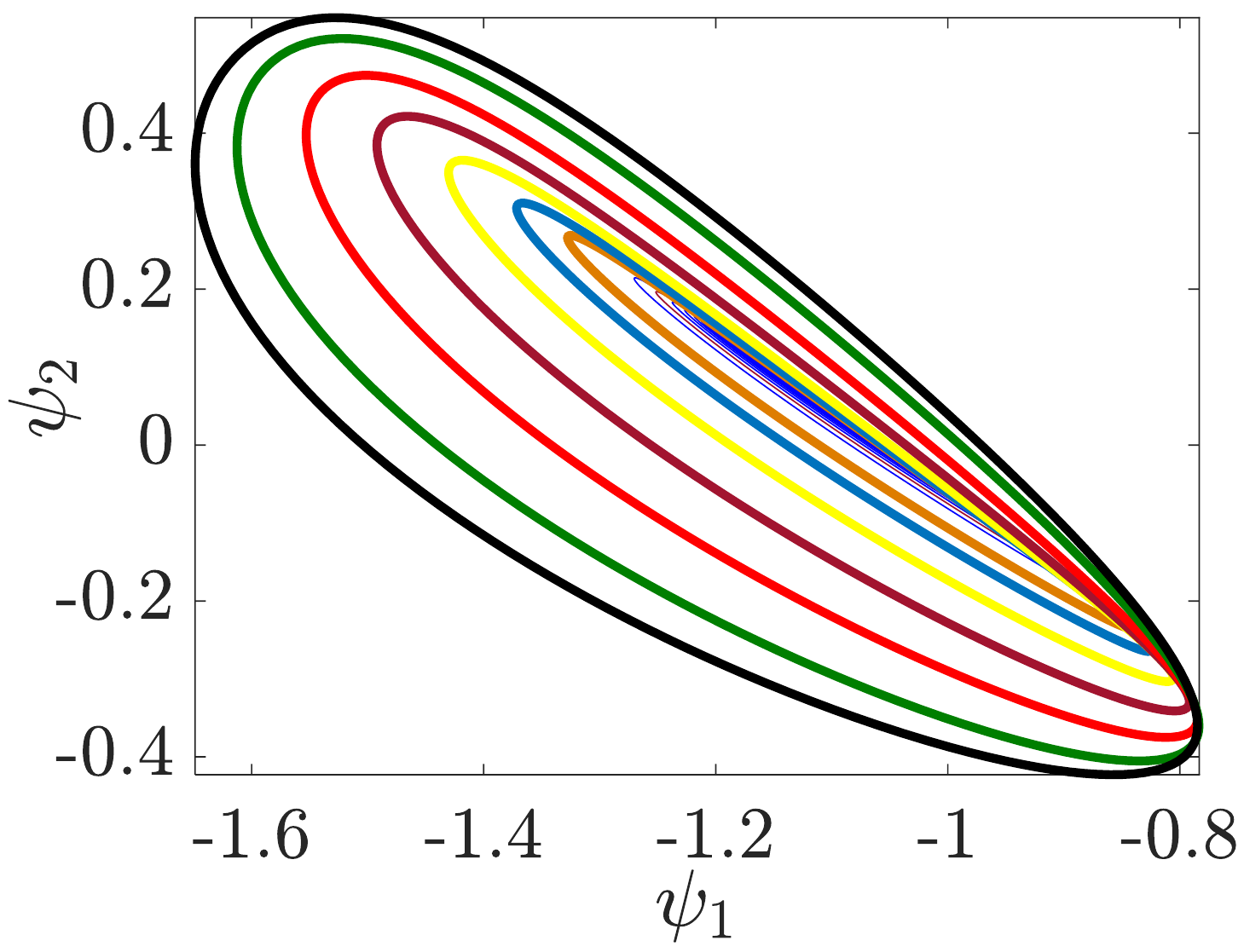}
  \caption{}
  \label{fig:sfig1 case1 bifurcation}
\end{subfigure}%
\begin{subfigure}{0.35\textwidth}
  \centering
  \includegraphics[width=\linewidth]{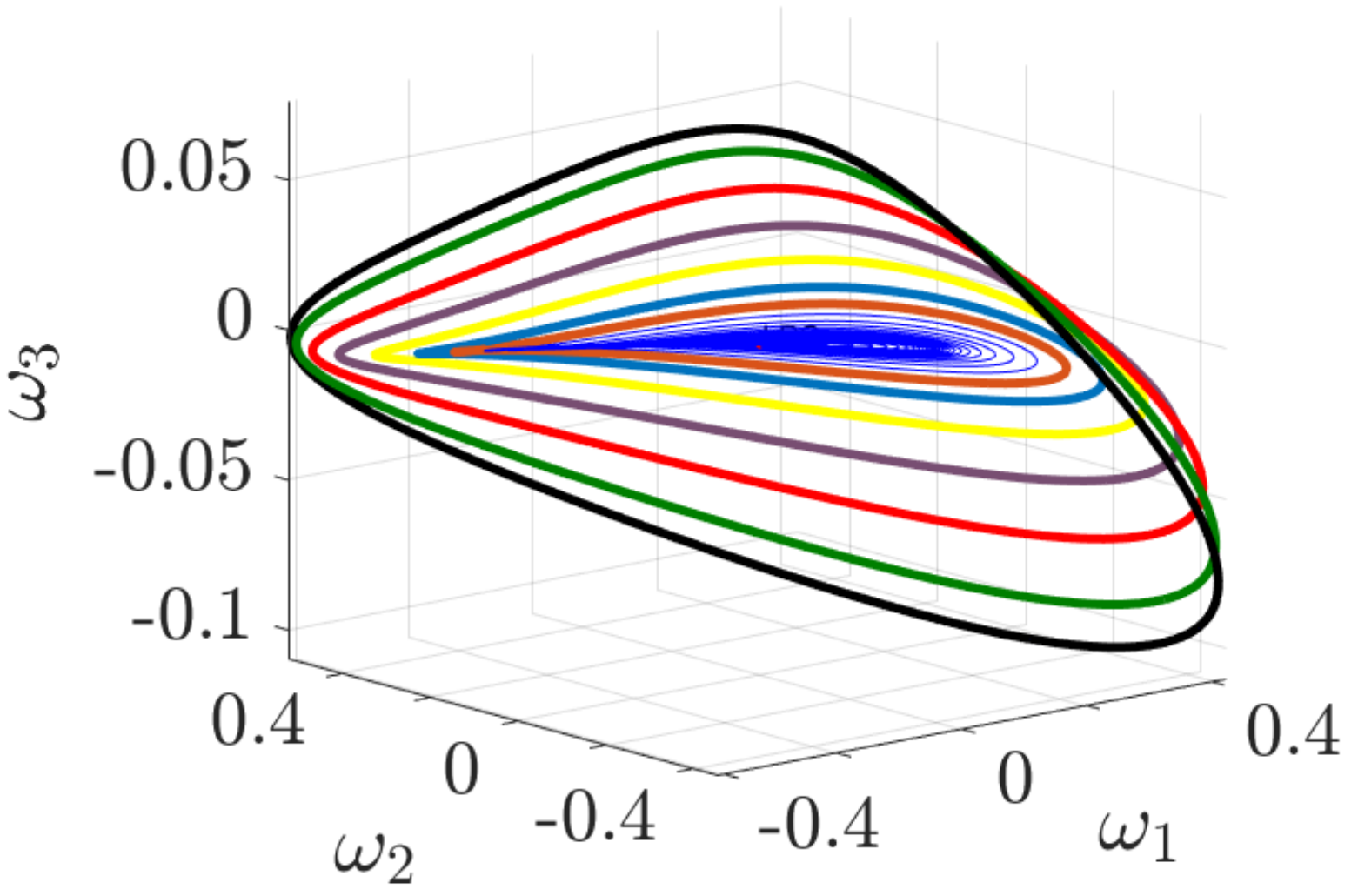}
  \caption{}
  \label{fig:sfig2 case1 bifurcation}
\end{subfigure}
\begin{subfigure}{0.34\textwidth}
  \centering
  \includegraphics[width=\linewidth]{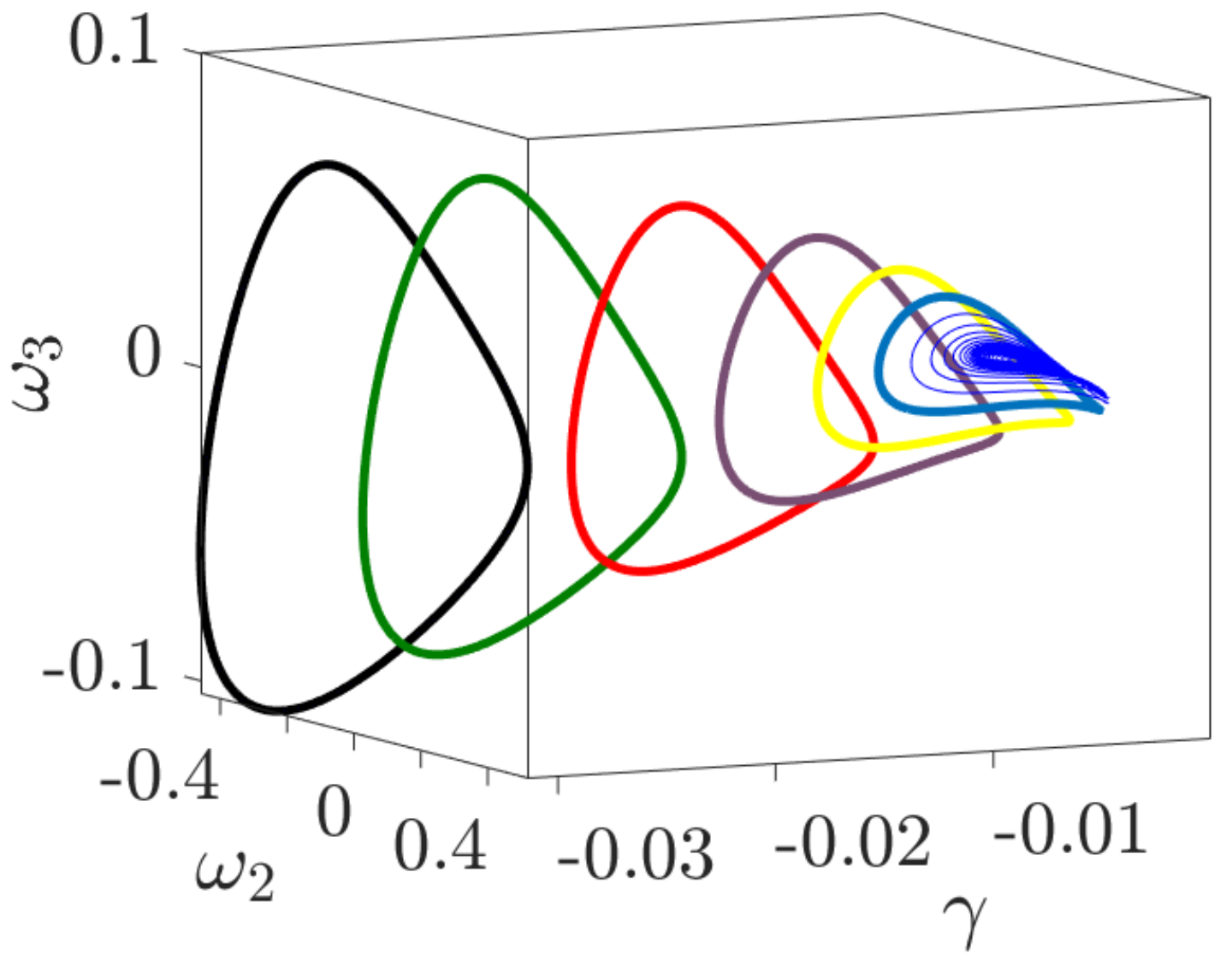}
  \caption{}
  \label{fig:sfig3 case1 bifurcation}
\end{subfigure}
\begin{subfigure}{0.35\textwidth}
  \centering
  \includegraphics[width=\linewidth]{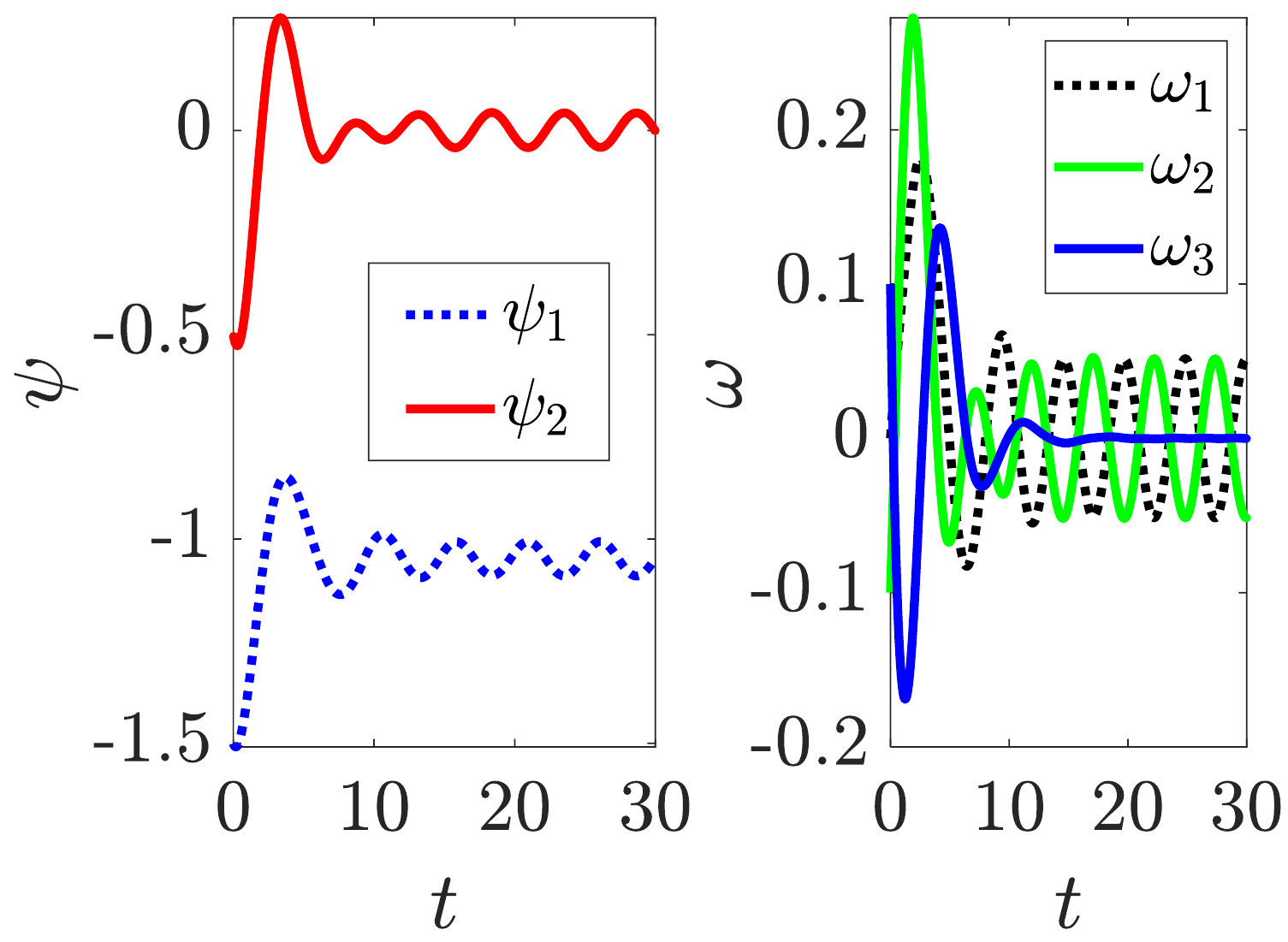}
  \caption{}
  \label{fig:sfig4 case1 bifurcation}
\end{subfigure}
\caption{Occurrence of supercritical Hopf bifurcation in Case $1$. (a)-(c) Projection of limit cycles into different subspaces as the parameter $\gamma$ changes. (d) Oscillations of the voltage angles $\psi$ in radians and the  angular frequency deviation $\omega$ in radians per seconds when $\gamma=0$. Note that $\psi_3 \equiv 0$.}
\label{fig: case1 bifurcation}
\end{figure}
\subsection{Case $2$} \label{subsubsec: case 2}
Next, we explore how damping variations could lead to a Hopf bifurcation in lossy systems. It is proved in \cref{prop:hyperbolicity n2 n3} that a $2$-generator system with only one undamped generator cannot experience a Hopf bifurcation. To complete the discussion, let us consider a fully-damped (i.e., all generators have nonzero damping) lossy $2$-generator system here. Note also that the discussion about a fully-undamped case is irrelevant (see \cref{lemma: undamped systems}). Suppose $M=I_2$, $D= \mathbf{diag} ([\gamma,1])$,  $Y_{12}=-1+\IU 5.7978$ p.u., $P_{m_1}=6.6991$ p.u., and $P_{m_2}=-4.8593$ p.u. The load flow problem for this system has the solution $V_j=1$ p.u. $\forall j$ and $\delta_1=1.4905$, $\delta_2=0$. 
We observe that $\gamma=0.2$ will break the hyperbolicity and lead to a Hopf bifurcation  with the first Lyapunov coefficient $l_1(0.2)=1.15$. This positive value for for $l_1(0.2)$ confirms that the type of Hopf bifurcation is subcritical and an unstable limit cycle bifurcates for $\gamma\ge0.2$. Therefore, the situation can be summarized as follows:
\begin{itemize}
    \item If $\gamma <0.2$, there exists one unstable equilibrium point.
    \item If $\gamma =0.2$, a subcritical Hopf bifurcation takes place and a unique small unstable limit cycle is born.
    \item If $ \gamma >0.2$, there exists a stable equilibrium point surrounded by an unstable limit cycle. 
\end{itemize}

Figs. (\ref{fig:sfig1 case2 bifurcation})-(\ref{fig:sfig3 case2 bifurcation}) depict the bifurcating unstable limit cycles when the parameter $\gamma$ changes in the interval $[0.2,0.35]$. This case study sheds lights on an important fact: bifurcation can happen even in fully damped systems, provided that the damping matrix $D$ reaches a critical point (say $D_c$). When $D \preceq D_c$, the equilibrium point is unstable. On the other hand, when $D \succ D_c$, the equilibrium point becomes stable but it is still surrounded by an unstable limit cycle. As we increase the damping parameter, the radius of the limit cycle increases, and this will enlarge the region of attraction of the equilibrium point. Note that the region of attraction of the equilibrium point is surrounded by the unstable limit cycle. This also confirms the monotonic behaviour of damping proved in \cref{thrm: monotonic damping behaviour}. Fig. (\ref{fig:sfig4 case2 bifurcation}) shows the region of attraction surrounded by the unstable limit cycle (in red) when $\gamma=0.25$. In this figure, the green orbits located inside the cycle are spiraling in towards the equilibrium point while the blue orbits located outside the limit cycle are spiraling out.
\begin{figure}
\centering
\begin{subfigure}{0.35\textwidth}
  \centering
  \includegraphics[width=\linewidth]{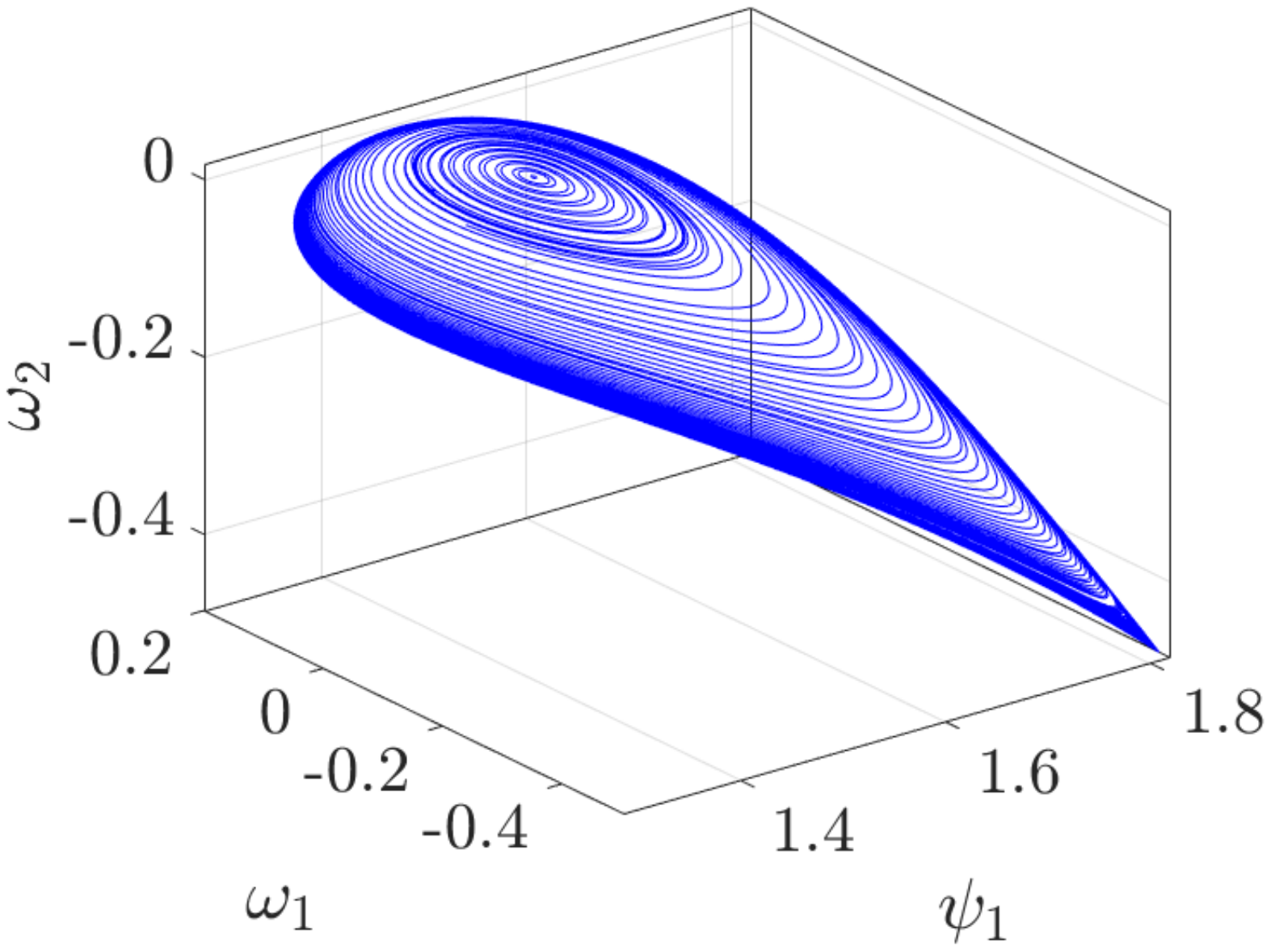}
  \caption{}
  \label{fig:sfig1 case2 bifurcation}
\end{subfigure}%
\begin{subfigure}{0.35\textwidth}
  \centering
  \includegraphics[width=0.9\linewidth]{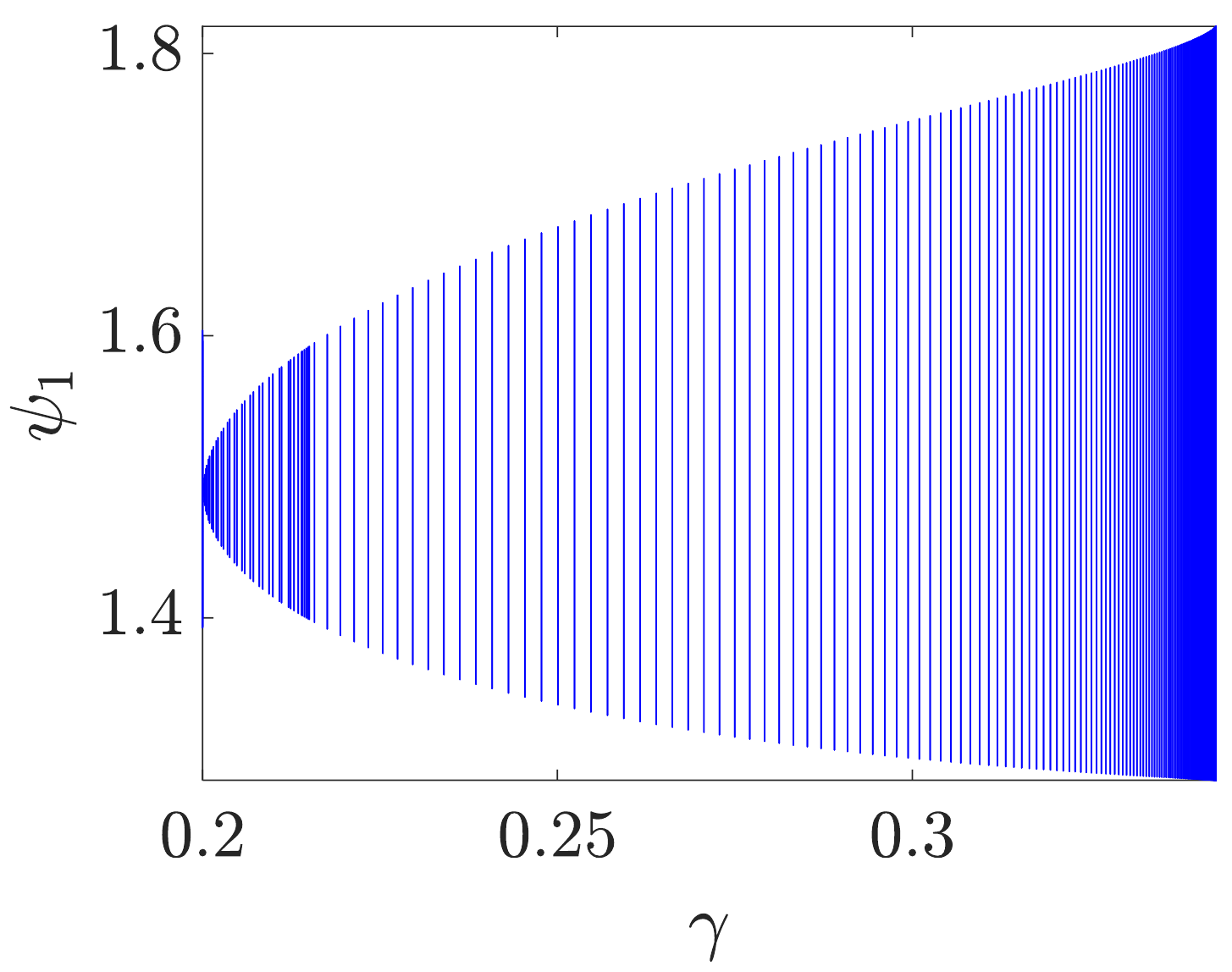}
  \caption{}
  \label{fig:sfig2 case2 bifurcation}
\end{subfigure}
\begin{subfigure}{0.35\textwidth}
  \centering
  \includegraphics[width=\linewidth]{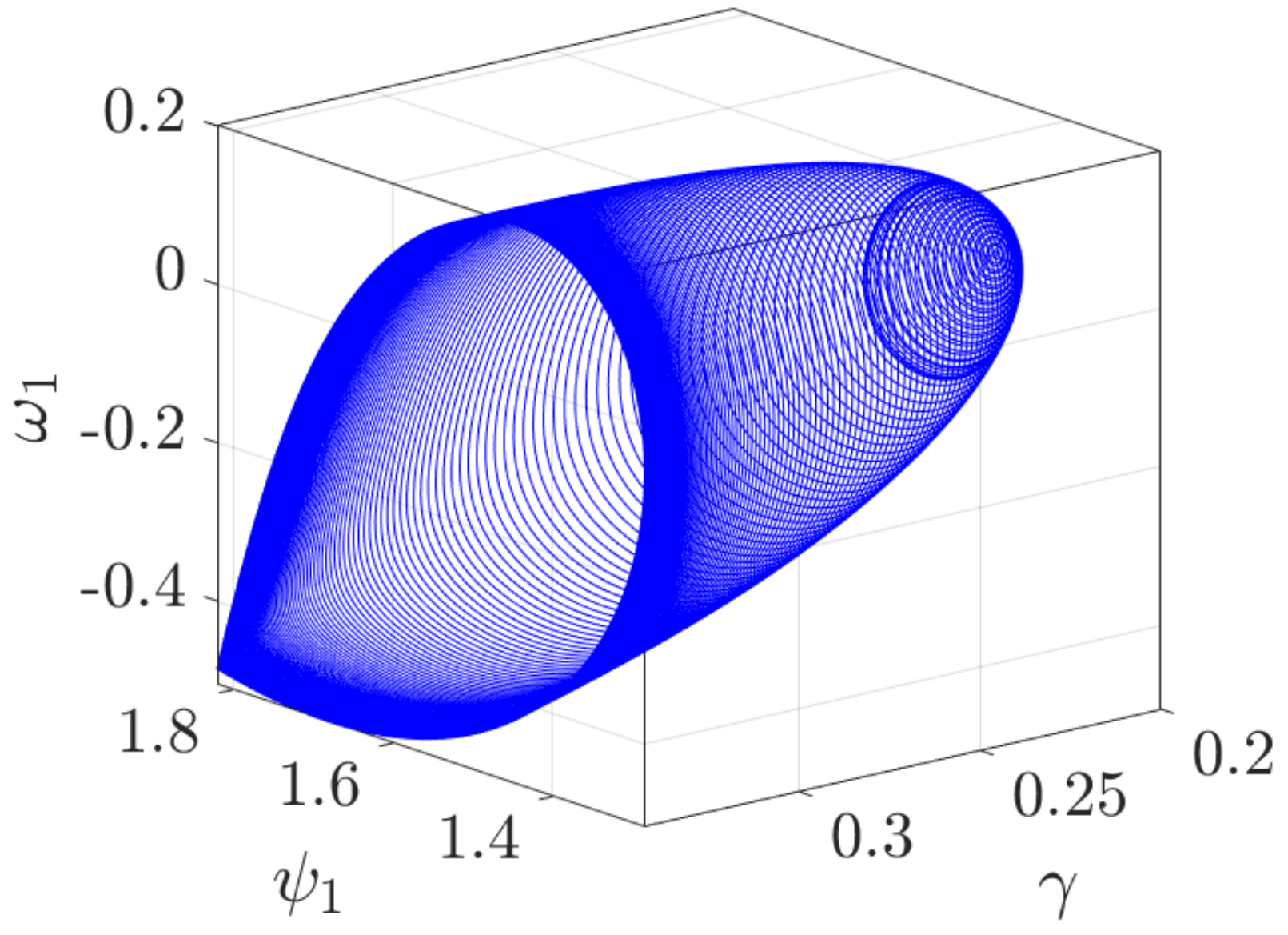}
  \caption{}
  \label{fig:sfig3 case2 bifurcation}
\end{subfigure}
\begin{subfigure}{0.34\textwidth}
  \centering
  \includegraphics[width=\linewidth]{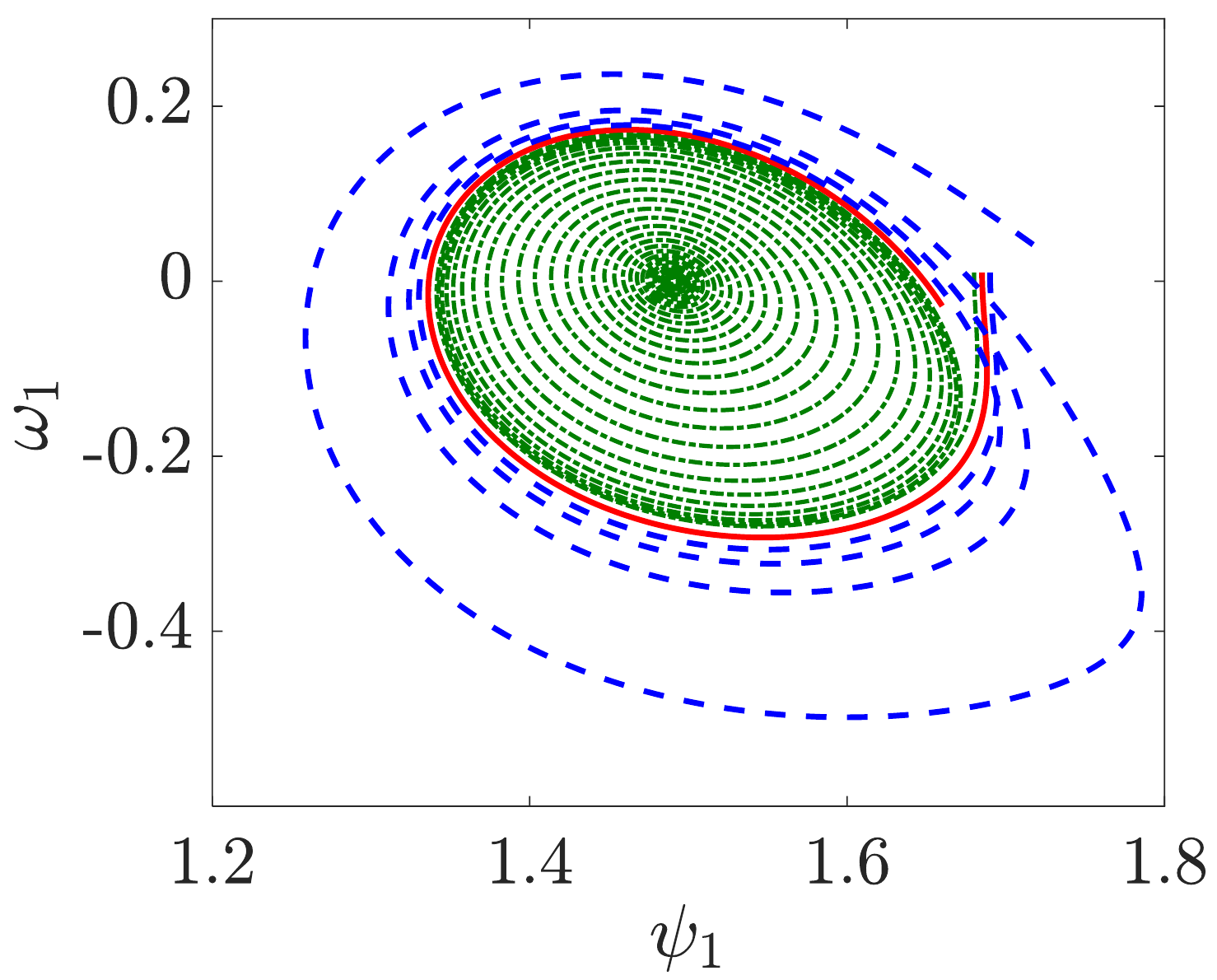}
  \caption{}
  \label{fig:sfig4 case2 bifurcation}
\end{subfigure}
\caption{Occurrence of subcritical Hopf bifurcation in Case $2$. (a) Unstable limit cycles as the parameter $\gamma$ changes. (b)-(c) Projection of limit cycles into different subspaces as the parameter $\gamma$ changes. (d) The region of attraction of the equilibrium point when $\gamma=0.25$. The unstable limit cycle is shown in red, while the orbits inside and outside of it are shown in green and blue, respectively. Note that $\psi_2 \equiv 0$.}
\label{fig: case2 bifurcation}
\end{figure}

Although both supercritical and subcritical Hopf bifurcations lead to the birth of limit cycles, they have quite different practical consequences. The supercritical Hopf bifurcation which occurred \Cref{subsubsec: case 1} corresponds to a soft or noncatastrophic stability loss because a stable equilibrium point is replaced with a stable periodic orbit, and the system remains in a neighborhood of the equilibrium. In this case, the system operator can take appropriate measures to bring the system back to the stable equilibrium point. Conversely, the subcritical Hopf bifurcation in \Cref{subsubsec: case 2} comes with a sharp or catastrophic loss of stability. This is because the region of attraction of the equilibrium point (which is bounded by the unstable limit cycle) shrinks as we decrease the parameter $\gamma$ and disappears once we hit $\gamma=0.2$. In this case, the system operator may not be able to bring the system back to the stable equilibrium point as the operating point may have left the region of attraction.

\section{Conclusions} \label{Sec: Conclusions}
We have presented a comprehensive study on the role of damping in a large class of dynamical systems, including electric power networks. Paying special attention to partially-damped systems, it is shown that damping plays a monotonic role in the hyperbolicity of the equilibrium points. We have proved that the hyperbolicity of the equilibrium points is intertwined with the observability of a pair of matrices, where the damping matrix is involved. We have also studied the aftermath of hyperbolicity collapse, and have shown both subcritical and supercritical Hopf bifurcations can occur as damping changes. It is shown that Hopf bifurcation cannot happen in small power systems with only one undamped generator. In the process, we have developed auxiliary results by proving some important spectral properties of  the power system Jacobian matrix, establishing the relationship between a power system model and its referenced counterpart, and finally addressing a fundamental question from matrix perturbation theory. Among others, the numerical experiments have illustrated how damping can change the region of attraction of the equilibrium points.
We believe our results are of general interest to the community of applied dynamical systems, and provide new insights into the interplay of damping and oscillation in one of the most important engineering system, the electric power systems.
\bibliographystyle{siamplain}
\bibliography{References}

\begin{thebibliography}{10}

\bibitem{nerc2019}
{\em Eastern interconnection forced oscillation event report on january 11,
  2019},
  \url{https://www.nerc.com/pa/rrm/ea/Documents/January_11_Oscillation_Event_Report.pdf}
  (accessed 8-30-2020).
\newblock North American Electric Reliability Corporation.

\bibitem{1981-Abed-Oscillations}
{\sc E.~Abed and P.~Varaiya}, {\em Oscillations in power systems via hopf
  bifurcation}, in 20th IEEE Conference on Decision and Control, IEEE, 1981,
  pp.~926--929.

\bibitem{2005-kuramoto-review}
{\sc J.~A. Acebr{\'o}n, L.~L. Bonilla, C.~J.~P. Vicente, F.~Ritort, and
  R.~Spigler}, {\em The kuramoto model: A simple paradigm for synchronization
  phenomena}, Rev. Modern Phys., 77 (2005), pp.~137--185.

\bibitem{2013-adhikari-structural-book}
{\sc S.~Adhikari}, {\em Structural Dynamic Analysis with Generalized Damping
  Models: Analysis}, John Wiley \& Sons, 2013.

\bibitem{1986-Alexander-Oscillatory-Solutions}
{\sc J.~Alexander}, {\em Oscillatory solutions of a model system of nonlinear
  swing equations}, Int. J. Elec. Power Ener. Syst., 8 (1986), pp.~130--136.

\bibitem{2008-anderson-stability}
{\sc P.~M. Anderson and A.~A. Fouad}, {\em Power System Control and Stability},
  John Wiley \& Sons, 2008.

\bibitem{1982-Arapostathis-global-analysis-periodicSol}
{\sc A.~Arapostathis, S.~Sastry, and P.~Varaiya}, {\em Global analysis of swing
  dynamics}, IEEE Trans. Circuits Syst., 29 (1982), pp.~673--679.

\bibitem{2016-Igor-Inertia-Kuramoto}
{\sc I.~V. Belykh, B.~N. Brister, and V.~N. Belykh}, {\em Bistability of
  patterns of synchrony in kuramoto oscillators with inertia}, Chaos, 26
  (2016), p.~094822.

\bibitem{2020-Igor-Inertia-Kuramoto}
{\sc B.~N. Brister, V.~N. Belykh, and I.~V. Belykh}, {\em When three is a
  crowd: Chaos from clusters of kuramoto oscillators with inertia}, Phys. Rev.
  E, 101 (2020), p.~062206.

\bibitem{1988-chiang-stability-of-nonlinear}
{\sc H.-D. Chiang and F.~F. Wu}, {\em Stability of nonlinear systems described
  by a second-order vector differential equation}, IEEE Trans. Circuits Syst.,
  35 (1988), pp.~703--711.

\bibitem{2013-Lee-Power-dynamics-stochastic}
{\sc S.~V. Dhople, Y.~C. Chen, L.~DeVille, and A.~D.
  Dom{\'\i}nguez-Garc{\'\i}a}, {\em Analysis of power system dynamics subject
  to stochastic power injections}, IEEE Trans. Circuits Syst. I: Regul. Pap.,
  60 (2013), pp.~3341--3353.

\bibitem{2011-dorfler-critical-coupling}
{\sc F.~Dorfler and F.~Bullo}, {\em On the critical coupling for kuramoto
  oscillators}, SIAM J. Appl. Dyn. Syst., 10 (2011), pp.~1070--1099.

\bibitem{2012-Dorfler-synchronization}
{\sc F.~Dorfler and F.~Bullo}, {\em Synchronization and transient stability in
  power networks and nonuniform kuramoto oscillators}, SIAM J. Control Optim.,
  50 (2012), pp.~1616--1642.

\bibitem{2020-fast-certificate}
{\sc A.~Gholami and X.~A. Sun}, {\em A fast certificate for power system
  small-signal stability}, in 59th IEEE Conference on Decision and Control,
  2020, pp.~3383--3388.
\newblock \href{https://arxiv.org/abs/2008.02263}{arXiv:2008.02263}.

\bibitem{2021-gholami-sun-MMG-stability}
{\sc A.~Gholami and X.~A. Sun}, {\em Stability of multi-microgrids: New
  certificates, distributed control, and braess's paradox}, IEEE Trans. Control
  Netw. Syst.,  (2021).
\newblock \href{https://arxiv.org/abs/2103.15308}{arXiv:2103.15308}.

\bibitem{2020-greenbaum-perturbation}
{\sc A.~Greenbaum, R.-c. Li, and M.~L. Overton}, {\em First-order perturbation
  theory for eigenvalues and eigenvectors}, SIAM Rev., 62 (2020), pp.~463--482.

\bibitem{2007-Hiskens-Limit-cycle2}
{\sc I.~A. Hiskens and P.~B. Reddy}, {\em Switching-induced stable limit
  cycles}, Nonlinear Dyn., 50 (2007), pp.~575--585.

\bibitem{2013-Horn-matrix-analysis}
{\sc R.~A. Horn and C.~R. Johnson}, {\em Matrix Analysis}, Cambridge University
  Press, New York, 2~ed., 2013.

\bibitem{2017-koerts-second-order}
{\sc F.~J. Koerts, M.~Burger, A.~J. van~der Schaft, and C.~D. Persis}, {\em
  Topological and graph-coloring conditions on the parameter-independent
  stability of second-order networked systems}, SIAM J. Control Optim., 55
  (2017), pp.~3750--3778.

\bibitem{2019-Patrick-koorehdavoudi-input}
{\sc K.~Koorehdavoudi, S.~Roy, T.~Prevost, F.~Xavier, P.~Panciatici, and V.~M.
  Venkatasubramanian}, {\em Input-output properties of the power grid’s swing
  dynamics: dependence on network parameters}, in IEEE Conference on Control
  Technology and Applications, 2019.

\bibitem{1996-blackout}
{\sc D.~N. Kosterev, C.~W. Taylor, and W.~A. Mittelstadt}, {\em Model
  validation for the august 10, 1996 wscc system outage}, IEEE Trans. Power
  Syst., 14 (1999), pp.~967--979.

\bibitem{1994-kundur-stability}
{\sc P.~Kundur}, {\em Power System Stability and Control}, New York, NY, USA:
  McGraw-Hill, 1994.

\bibitem{2004-kuznetsov-hopf}
{\sc Y.~A. Kuznetsov}, {\em Elements of Applied Bifurcation Theory}, vol.~112,
  Springer Science \& Business Media, 2004.

\bibitem{1990-Kwatny-Frequency-Analysis_Hopf}
{\sc H.~Kwatny and G.~Piper}, {\em Frequency domain analysis of hopf
  bifurcations in electric power networks}, IEEE Trans. Circuits Syst., 37
  (1990), pp.~1317--1321.

\bibitem{1989-Kwatny-Energy-Analysis-Flutter}
{\sc H.~G. Kwatny and X.-M. Yu}, {\em Energy analysis of load-induced flutter
  instability in classical models of electric power networks}, IEEE Trans.
  Circuits Syst., 36 (1989), pp.~1544--1557.

\bibitem{1984-laub-controllability}
{\sc A.~Laub and W.~Arnold}, {\em Controllability and observability criteria
  for multivariable linear second-order models}, IEEE Trans. Automat. Control,
  29 (1984), pp.~163--165.

\bibitem{2010-Ma-decoupling}
{\sc F.~Ma, M.~Morzfeld, and A.~Imam}, {\em The decoupling of damped linear
  systems in free or forced vibration}, J. Sound Vib., 329 (2010),
  pp.~3182--3202.

\bibitem{1980-Miller-Asymptotic}
{\sc R.~Miller and A.~Michel}, {\em Asymptotic stability of systems: results
  involving the system topology}, SIAM J. Control Optim., 18 (1980),
  pp.~181--190.

\bibitem{2012-aminifar-wide-area-damping}
{\sc M.~Mokhtari, F.~Aminifar, D.~Nazarpour, and S.~Golshannavaz}, {\em
  Wide-area power oscillation damping with a fuzzy controller compensating the
  continuous communication delays}, IEEE Trans. Power Syst., 28 (2012),
  pp.~1997--2005.

\bibitem{2005-ortega-transient}
{\sc R.~Ortega, M.~Galaz, A.~Astolfi, Y.~Sun, and T.~Shen}, {\em Transient
  stabilization of multimachine power systems with nontrivial transfer
  conductances}, IEEE Trans. Automat. Control, 50 (2005), pp.~60--75.

\bibitem{2017-paganini-global}
{\sc F.~Paganini and E.~Mallada}, {\em Global performance metrics for
  synchronization of heterogeneously rated power systems: The role of machine
  models and inertia}, in 55th Annual Allerton Conference on Communication,
  Control, and Computing, 2017.

\bibitem{1974-M-matrices}
{\sc G.~Poole and T.~Boullion}, {\em A survey on {M}-matrices}, SIAM Rev., 16
  (1974), pp.~419--427.

\bibitem{2005-Hiskens-limit-cycle1}
{\sc P.~B. Reddy and I.~A. Hiskens}, {\em Limit-induced stable limit cycles in
  power systems}, IEEE Russia Power Tech, 2005.

\bibitem{1978-Schmidt-hopf}
{\sc D.~S. Schmidt}, {\em Hopf's bifurcation theorem and the center theorem of
  liapunov with resonance cases}, J. Math. Anal. Appl., 63 (1978),
  pp.~354--370.

\bibitem{1980-skar-nontrivial-transfer-conductance1}
{\sc S.~J. Skar}, {\em Stability of multi-machine power systems with nontrivial
  transfer conductances}, SIAM J. Appl. Math., 39 (1980), pp.~475--491.

\bibitem{1980-Skar-stability-thesis}
{\sc S.~J. Skar}, {\em Stability of power systems and other systems of second
  order differential equations}, PhD thesis, Dept. Math., Iowa State Univ.,
  Iowa, USA, 1980.

\bibitem{2001-Tisseur-Pencil}
{\sc F.~Tisseur and K.~Meerbergen}, {\em The quadratic eigenvalue problem},
  SIAM Rev., 43 (2001), pp.~235--286.

\bibitem{2016-Vu-framework}
{\sc T.~L. Vu and K.~Turitsyn}, {\em A framework for robust assessment of power
  grid stability and resiliency}, IEEE Trans. Automat. Control, 62 (2016),
  pp.~1165--1177.

\bibitem{2018-Dorfler-robust}
{\sc E.~Weitenberg, Y.~Jiang, C.~Zhao, E.~Mallada, C.~De~Persis, and
  F.~D{\"o}rfler}, {\em Robust decentralized secondary frequency control in
  power systems: Merits and tradeoffs}, IEEE Trans. Automat. Control, 64
  (2018), pp.~3967--3982.

\bibitem{2014-Low-NaLi-stability}
{\sc C.~Zhao, U.~Topcu, N.~Li, and S.~Low}, {\em Design and stability of
  load-side primary frequency control in power systems}, IEEE Trans. Automat.
  Control, 59 (2014), pp.~1177--1189.

\end{thebibliography}
\appendix
%
%
%
\section{Proof of \Cref{lemma: rank principal}} \label{proof of lemma: rank principal}
\begin{proof}
    	Assume that all $r$-by-$r$ principal submatrices of $S$ are singular, and let us lead this assumption to a contradiction. Since $\rank(S)=r$, all principal submatrices of size larger than $r$ are also singular. Therefore, zero is an eigenvalue of every $m$-by-$m$ principal submatrix of $S$ for each $m\ge r$. Consequently, all principal minors of $S$ of size $m$ are zero for each $m\ge r$. Let $E_\ell(S)$ denote the sum of principal minors of $S$ of size $\ell$ (there are $n\choose \ell$ of them), and observe that we have $E_m(S) = 0, \forall m \ge r$. Moreover, thought of as a formal polynomial in $t$, let $p_S (t)=\sum_{\ell=0}^n a_\ell t^\ell$ with $a_n=1$ be the characteristic polynomial of $S$, and recall that the $k$-th derivative of $p_S(t)$ at $t=0$ is $p_S^{(k)}(0)=k!(-1)^{n-k}E_{n-k}(S), \forall k\in\{0,1,\cdots,n-1\}$, and the coefficients of the characteristic polynomial are $a_k=\frac{1}{k!}p_S^{(k)}(0)$. 
    	In this case, our assumption leads to $a_k=p_S^{(k)}(0)=0,\forall k\in\{0,1,\cdots,n-r\}$, i.e., zero is an eigenvalue of $S$ with algebraic multiplicity at least $n-r+1$. But from the assumption of the lemma we know $S$ is similar to $B\oplus0_{n-r}$, that is, zero is an eigenvalue of $S$ with algebraic multiplicity exactly $n-r$, and we arrive at the desired contradiction.
\end{proof}



\section{Stability of Symmetric Second-Order Systems with Nonsingular Damping}
\label{appendix: Stability of Symmetric Second-Order Systems with Nonsingular Damping}

\Cref{thrm: nec and suf for pure imaginary lossless} provides a necessary and sufficient condition for the hyperbolicity of an equilibrium point $(x_0,0)$ of the second-order system \eqref{eq: nonlinear ode}, when the inertia, damping, and Jacobian of $f$ satisfy $M\in\S^{n}_{++}, D\in\S^n_+, \nabla f(x_0)\in\S^n_{++}$. 
In this section, we prove that if we replace the assumption $D\in\S^n_+$ with $D\in\S^n_{++}$, then the equilibrium point $(x_0,0)$ is not only hyperbolic but also asymptotically stable. This asymptotic stability is proved for lossless swing equations in \cite[Theorem 1, Part d]{2021-gholami-sun-MMG-stability}.
The next theorem generalizes \cite[Theorem 1, Part d]{2021-gholami-sun-MMG-stability} to the second-order system \eqref{eq: nonlinear ode} where the damping and inertia matrices are not necessarily diagonal.
\begin{theorem}[stability in second-order systems: symmetric case] \label{thrm: Stability of Symmetric Second-Order Systems with nonsing damping}
    Consider the second-order ODE system \eqref{eq: nonlinear ode} with inertia matrix $M\in\mathbb{S}^n_{++}$ and damping matrix $D \in\mathbb{S}^n_{++}$. Suppose $(x_0,0)\in \mathbb{R}^{n+n}$ is an equilibrium point of the corresponding first-order system \eqref{eq: nonlinear ode 1 order} with the Jacobian matrix $J\in\mathbb{R}^{2n\times 2n}$ defined in \eqref{eq: J general case} such that $L=\nabla f(x_0)\in \mathbb{S}^n_{++}$. Then, the equilibrium point $(x_0,0)$ is locally asymptotically stable.
\end{theorem}

\begin{proof}
We complete the proof in three steps:\newline
\textbf{Step 1:} First, we show all real eigenvalues
of $J$ are negative. Assume $\lambda\in\R, \lambda\ge0$ is a nonnegative eigenvalue of $J(x_0)$, and let us lead this assumption to a contradiction. According to \Cref{lemma: relation between ev J and ev J11},
\begin{align} \label{eq: pencil det zero appendix}
    \det\left(\lambda^2 M + \lambda D + L \right)  = 0.
\end{align}
Since all three matrices $L, D$, and $M$ are positive definite, the
matrix pencil $P(\lambda)=\lambda^2 M + \lambda D + L $ is also a positive definite matrix for any nonnegative $\lambda$. Hence $P(\lambda)$ is nonsingular, contradicting \eqref{eq: pencil det zero appendix}.

\noindent
\textbf{Step 2:} Next, we prove that the eigenvalues of $J$ cannot be purely imaginary. We provide two different proofs for this step.
According to our assumption, the damping matrix $D$ is nonsingular, and the pair $(M^{-1}\nabla f(x_0),M^{-1}D)$ is always observable because the nullspace of $M^{-1}D$ is trivial. Hence, according to
 \Cref{thrm: nec and suf for pure imaginary lossless}, the  equilibrium point $(x_0,0)$ is hyperbolic, and $J(x_0)$ does not have any purely imaginary eigenvalue, so the first proof of this step is complete. For the second proof,
let $\lambda \in \sigma(J(x_0))$, then according to \Cref{lemma: relation between ev J and ev J11}, $\exists v\in\C^n, v\ne0$ such that $(\lambda^2 M + \lambda D + L) v = 0.$
	Suppose, for the sake of contradiction, that $\lambda  = \IU \beta \in \sigma(J(x_0))$ for some nonzero real $\beta$. Let $v  =  x + \IU y$, then $((L - \beta^2 M) + \IU\beta D) (x+\IU y) = 0$,
	which can be equivalently written as
	\begin{align}
	\begin{bmatrix}
	L - \beta^2 M & -\beta D \\
	\beta D & L - \beta^2 M
	\end{bmatrix} \begin{bmatrix}
	x\\y
	\end{bmatrix} = \begin{bmatrix}
	0\\0
	\end{bmatrix}.
	\end{align}
	Define the matrix 
	\begin{align} \label{eq:M}
	H(\beta) := \begin{bmatrix}
	\beta D &  L  - \beta^2 M\\
	 L  - \beta^2 M & -\beta D
	\end{bmatrix}. 
	\end{align}
	Since $L \in \mathbb{S}^n_{++}$, $H(\beta)$ is a symmetric matrix.
	Notice also that $H(\beta)$ cannot be positive semidefinite due to the diagonal blocks $\pm\beta D$. 
	Since $D\in \mathbb{S}^n_{++}$, the determinant of $H(\beta)$ can be expressed using Schur complement as
	\begin{align*}
	 {\det}(H(\beta))  =  {\det} (-\beta D) {\det} (\beta D 
	  + \beta^{-1} ( L  - \beta^2 M )D^{-1}(  L  - \beta^2 M)).
	\end{align*}
	So we only need to consider the nonsingularity of the Schur complement.
	Define the following matrices for the convenience of analysis: 
	\begin{align*}
		A(\beta) &:=  L  - \beta^2 M, \\
		B(\beta) &:= D^{-\frac{1}{2}}A(\beta)D^{-\frac{1}{2}},\\
		E(\beta) &:= I + \beta^{-2} B(\beta)^2.
	\end{align*}
	The inner matrix of the Schur complement can be written as
		\begin{align*}
		& \beta D + \beta^{-1} ( L  - \beta^2 M )D^{-1}(  L  - \beta^2 M) \\ 
	 & =  \beta D^{\frac{1}{2}} (I + \beta^{-2} D^{-\frac{1}{2}}A(\beta)D^{-1}A(\beta)D^{-\frac{1}{2}})D^{\frac{1}{2}} \\
	 & =  \beta D^{\frac{1}{2}} (I + \beta^{-2} B(\beta)^2)D^{\frac{1}{2}} = \beta D^{\frac{1}{2}} E(\beta) D^{\frac{1}{2}}.
		\end{align*}
	Notice that {$E(\beta)$ and $B(\beta)$ have the same eigenvectors and the eigenvalues of $E(\beta)$ and $B(\beta)$ have a one-to-one correspondence: $\mu$ is an eigenvalue of $B(\beta)$ if and only if $1 + \beta^{-2} \mu^2$ is an eigenvalue of $E(\beta)$}. Indeed, we have 
	 $E(\beta)v = v + \beta^{-2}B(\beta)^2v = v + \beta^{-2}\mu^2 v = (1+\beta^{-2}\mu^2)v$ for any eigenvector $v$ of $B(\beta)$ with eigenvalue $\mu$. 
	Since $B(\beta)$ is symmetric, $\mu$ is a real number. Hence, $E(\beta)$ is positive definite (because $1+\beta^{-2}\mu^2>0$), therefore $H(\beta)$ is nonsingular for any real nonzero $\beta$. Then, the eigenvector $v=x+\IU y$ is zero which is a contradiction. This proves that $J(x_0)$ has no eigenvalue on the punctured imaginary axis.

	
	\noindent \textbf{Step 3:}
	Finally, we prove that any complex nonzero eigenvalue of $J(x_0)$ has a negative real part. 
	 For a complex eigenvalue $\alpha + \IU \beta$ of $J(x_0)$ with $\alpha\ne 0, \beta\ne 0$, by setting $v  =  x + \IU  y$, the pencil singularity equation becomes
	\begin{align*}
	( L  + (\alpha + \IU \beta) D + (\alpha^2 - \beta^2 + 2\alpha\beta \IU )M)(x+\IU y) = 0.
	\end{align*}
	Similar to Step 2 of the proof, define the matrix $H(\alpha,\beta)$ as
	\begin{align*}
	H(\alpha,\beta):=\begin{bmatrix}
	 L  + \alpha D + (\alpha^2-\beta^2) M & -\beta(D + 2\alpha M) \\
	\beta(D+2\alpha M) &  L  + \alpha D + (\alpha^2-\beta^2) M
	\end{bmatrix}.
	\end{align*}
	We only need to consider two cases, namely 1) $\alpha > 0, \beta > 0$ or 2) $\alpha <0, \beta > 0$. For the first case, $\beta(D+{2}\alpha M)$ is invertible and positive definite, therefore, we only need to look at the invertibility of the Schur complement
	\begin{align*}
	& S(\alpha,\beta) + T(\alpha,\beta) S^{-1}(\alpha,\beta) T(\alpha,\beta),
	\end{align*}
	where $S(\alpha,\beta):= \beta(D+{2}\alpha M)$ and $T(\alpha,\beta):= L  + \alpha D + (\alpha^2-\beta^2) M$.
	Using the same manipulation as in Step 1 of the proof, we can see that the Schur complement is always invertible for any $\alpha >0, \beta>0$. This implies the eigenvector $v$ is 0, which is a contradiction. Therefore, the first case is not possible. So any complex nonzero eigenvalue of $J(x_0)$ has a negative real part.
\end{proof}

\section{Proof of \cref{thrm: nec and suf for pure imaginary lossy}}
\label{proof of thrm: nec and suf for pure imaginary lossy}
\begin{proof}
    %
	There exist	$\lambda \in \mathbb{R_{+}}, \lambda\ne0$ and $x \in \mathbb{C}^n , x\ne 0$ such that
	\begin{align} \label{eq: observability in lossy nonsymmetric}
	 M^{-1}Lx = \lambda x \text{ and } M^{-1}Dx = 0.
	\end{align}
	Define $\xi=\sqrt{-\lambda}$, which is a purely imaginary number. The quadratic matrix pencil $M^{-1}P(\xi)=\xi^2 I + \xi M^{-1}D + M^{-1}L$ is singular because $M^{-1}P(\xi)x = \xi^2 x + \xi M^{-1}Dx + M^{-1}Lx = -\lambda x + 0 + \lambda x = 0$. By  \cref{lemma: relation between ev J and ev J11}, $\xi$ is an eigenvalue of $J$. Similarly, we can show $-\xi$ is an eigenvalue of $J$. Therefore, $\sigma(J)$ contains a pair of purely imaginary eigenvalues.
\end{proof}

\section{Proof of \cref{thrm: original model vs referenced model}}
\label{proof of thrm: original model vs referenced model}

Let us first prove the following useful lemma.
\begin{lemma} \label{lemma: pencil for referenced systems}
    Let $(\delta^0,\omega^0)$ be an equilibrium point of the swing equation \eqref{eq: swing equations} and $\Psi(\delta^0,\omega^0)$ be the corresponding equilibrium point of the referenced model \eqref{eq: Swing Equation Polar referenced}. Let $J^r$ denote the Jacobian of the referenced model at this equilibrium point.
	For any $\lambda \ne 0$, $\lambda$ is an eigenvalue of $J^r$ if and only if the quadratic matrix pencil $P(\lambda):= \lambda^2 M  + \lambda D + \nabla P_e (\delta^0)$ is singular.
\end{lemma}
\begin{proof}
 The referenced model \eqref{eq: Swing Equation Polar referenced} can be written as
\begin{align}\label{eq:swing reduced}
\begin{bmatrix}
\dot {\psi} \\
\dot \omega 
\end{bmatrix}
=
\begin{bmatrix}
T_1 \omega \\
- D M^{-1} \omega + M^{-1} (P_m -P_e^r(\psi))
\end{bmatrix}.
\end{align}
%
Note that the Jacobian of the referenced flow function $\nabla P^r_e(\psi)$ is an $n \times (n-1)$ matrix and we have $\nabla P^r_e(\psi^0) = \nabla P_e (\delta^0) T_2$, where
\begin{align}
T_2:=\begin{bmatrix}
I_{n-1} \\ 0_{1\times(n-1)}
\end{bmatrix} \in \mathbb{R}^{n\times(n-1)}.
\end{align}
Accordingly, the Jacobin of the referenced model \eqref{eq: Swing Equation Polar referenced} is
\begin{align}
J^r=\begin{bmatrix}
0_{(n-1)\times (n-1)} & T_1 \\
-M^{-1} \nabla P_e (\delta^0) T_2 & -DM^{-1}
\end{bmatrix}.
\end{align}
	\textit{Necessity:} Let $\lambda$ be a nonzero eigenvalue of $J^r$ and $ (v_1 , v_2 ) $ be the corresponding eigenvector with $v_1\in\mathbb{C}^{n-1}$ and $v_2\in\mathbb{C}^{n}$. Then	
	\begin{align} \label{eq: J cha eq referenced}
	\begin{bmatrix}
	0_{(n-1)\times (n-1)} & T_1 \\
	-M^{-1}\nabla P_e (\delta^0) T_2 & -DM^{-1}
	\end{bmatrix}   \begin{bmatrix} v_1 \\v_2  \end{bmatrix}  = \lambda   \begin{bmatrix} v_1 \\v_2  \end{bmatrix},
	\end{align}
	which implies that $ T_1 v_2 = \lambda v_1$. Since $\lambda \ne 0$, we can substitute $ \lambda^{-1} T_1 v_2 = v_1$ in the second equation to obtain 
	\begin{align} 
	\left(\lambda^2 M + \lambda D + \nabla P_e (\delta^0) T_2 T_1 \right) v_2 = 0. \label{eq: quadratic matrix pencil referenced}
	\end{align}  
	Since the eigenvector $(v_1,v_2)$ is nonzero, we have $v_2 \not = 0$ (otherwise $ v_1=\lambda^{-1} T_1 0 = 0   \implies (v_1,v_2) = 0 $), Eq. \eqref{eq: quadratic matrix pencil referenced} implies that the matrix pencil $P(\lambda)= \lambda^2 M + \lambda D + \nabla P_e (\delta^0) T_2 T_1$ is singular. Next, we show that $\nabla P_e (\delta^0) T_2 T_1 = \nabla P_e (\delta)$. Since $\nabla P_e (\delta^0)$ has zero row sum, it can be written as
	\begin{align*}
	\nabla P_e (\delta^0)=\begin{bmatrix}
	A&b\\c^\top&d
	\end{bmatrix}, \text{ where } A\mathbf{1}=-b, c^\top \mathbf{1}=-d.
	\end{align*}
	Therefore, we have
	\begin{align*}
	\nabla P_e (\delta^0) T_2 T_1 =  \begin{bmatrix}
	A&b\\c^\top&d
	\end{bmatrix} 
	\begin{bmatrix}
	I_{n-1} \\ 0
	\end{bmatrix}
	\begin{bmatrix}
	I_{n-1}& -\mathbf{1}
	\end{bmatrix}
	= 
	\begin{bmatrix}
	A&-A\mathbf{1}\\c^\top&-c^\top\mathbf{1}
	\end{bmatrix}
	=
	\begin{bmatrix}
	A&b\\c^\top&d
	\end{bmatrix}. 
	\end{align*}
	
	\textit{Sufficiency:}	
	Suppose there exists $\lambda \in \mathbb{C}, \lambda \ne 0$ such that  $P(\lambda)= \lambda^2 M + \lambda D + \nabla P_e (\delta^0)$ is singular. Choose a nonzero $v_2 \in \ker (P(\lambda))$ and let $v_1:=\lambda^{-1} T_1 v_2$. 
	Accordingly, the characteristic equation \eqref{eq: J cha eq referenced} holds, and consequently, $\lambda$ is a nonzero eigenvalue of $J^r$.
\end{proof}

Now, we are ready to prove \cref{thrm: original model vs referenced model}.

\begin{proof}
Any equilibrium point $(\delta^0,\omega^0)$ of the swing equation model \eqref{eq: swing equations} is contained in the set
\begin{align*}
    \mathcal{E}:=\left \{ (\delta,\omega)\in \mathbb{R}^{2n} : \omega = 0, P_{m_j} = P_{e_j}(\delta), \quad \forall j \in \{1,...,n\} \right \}.
\end{align*}
Let $(\psi^0,\omega^0)=\Psi(\delta^0,\omega^0)$, and note that $\omega^0=0$. From \eqref{eq: flow function} and \eqref{eq: flow function referenced compact}, we observe that $P_{e_j}(\delta^0) = P_{e_j}^r(\psi^0), \forall j \in \{1,...,n\}$ where $\psi^0_n :=0$. Therefore, $(\psi^0,\omega^0)$ is an equilibrium point of the the referenced model \eqref{eq: Swing Equation Polar referenced}. 
\\
To prove the second part, recall that $\lambda$ is an eigenvalue of the Jacobian of \eqref{eq: swing equations} at $(\delta^0,\omega^0)$ if and only if $\det( \nabla P_e (\delta^0)  +\lambda D + \lambda^2 M)=0$. 
According to \cref{lemma: pencil for referenced systems}, the nonzero eigenvalues $J$ and $J^r$ are the same. Moreover, the referenced model \eqref{eq: Swing Equation Polar referenced} has one dimension less than the  swing equation model \eqref{eq: swing equations}. This completes the proof.
%
\end{proof}


\section{Proof of \Cref{prop:hyperbolicity n2 n3}}
\label{proof of prop:hyperbolicity n2 n3}

We prove the following lemmas first:
\begin{lemma} \label{lemma: complex representation}
Let $A,B\in\mathbb{R}^{n\times n}$ and define 
$$C:=  \small{ \left[\begin{array}{cc}A&-B\\B&A\end{array}\right] }.$$
Then $\rank(C)=2\rank(A+\IU B)$ which is an even number.
\end{lemma}
\begin{proof}
Let $V:= \frac{1}{\sqrt{2}} \left[\begin{array}{cc}I_n&\IU I_n\\\IU I_n&I_n\end{array}\right] $ and observe that $V^{-1}=\bar{V}=V^*$, where $\bar{V}$ stands for the entrywise conjugate and $V^*$ denotes the conjugate transpose of $V$. We have 
\begin{align*}
V^{-1}CV=\left[\begin{array}{cc}A-\IU B&0\\0&A+\IU B\end{array}\right] = (A-\IU B) \oplus (A+\IU B).
\end{align*}
Since rank is a similarity invariant, we have $\rank(C)=\rank((A-\IU B) \oplus (A+\IU B))=2\rank(A+\IU B)$.
\end{proof}
\begin{lemma}  \label{lemma: matrix form of pencil sinularity}
$\lambda = \IU  \beta$ is an eigenvalue of $J$ if and only if
the matrix
\begin{align*}
\mathcal{M}(\beta):= \begin{bmatrix}
L-\beta^2 M & -\beta D\\ \beta  D & L-\beta^2 M  
\end{bmatrix}
\end{align*}
is singular. Here $L=\nabla P_e (\delta^0)$.
\end{lemma}
\begin{proof}
According to \Cref{lemma: relation between ev J and ev J11}, $\IU  \beta \in \sigma(J)$ if and only if $\exists x\in\mathbb{C}^n,x\ne0$ such that
    \begin{align} \label{eq: pencil matrix singularity theorem proof}
        \left( L - \beta^2 M + \IU  \beta D  \right) x = 0.
    \end{align}
Define $A:=L - \beta^2 M$, $B:=\beta D$, and let $x=u+\IU v$. Rewrite \eqref{eq: pencil matrix singularity theorem proof} as $(A+\IU  B)(u+\IU v) = (Au - Bv) + \IU (Av + Bu) = 0$,
which is equivalent to 
\begin{align*}
\begin{bmatrix}
A & -B \\ B & A 
\end{bmatrix} \begin{bmatrix}
u \\ v
\end{bmatrix} = 0.
\end{align*}
\end{proof}
Now, we are ready to prove \Cref{prop:hyperbolicity n2 n3}:
According to \Cref{lemma: matrix form of pencil sinularity}, $\IU \beta\in\sigma(J)$ for some nonzero real $\beta$ if and only if the matrix
\begin{align*} 
\mathcal{M}(\beta) := \begin{bmatrix}
L - \beta^2 M & -\beta D\\
\beta D & L - \beta^2 M
\end{bmatrix} 
\end{align*}
is singular. Recall that $L:=\nabla P_e (\delta^0)$. In the sequel, we will show under the assumptions of  \Cref{prop:hyperbolicity n2 n3}, $\mathcal{M}(\beta)$ is always nonsingular. First, we prove the theorem for $n=2$. In this case, 
$$
L= \left[ \begin{array}{cc}a_{12}&-a_{12}\\-a_{21}&a_{21}\end{array}\right], a_{12}>0, a_{21}>0.
$$
According to \Cref{lemma: complex representation}, we have $\rank(\mathcal{M}(\beta))=2\rank(L-\beta^2M- \IU  \beta D)$, and $L-\beta^2M- \IU  \beta D$ is full rank because 
\begin{align*}
L-\beta^2M- \IU  \beta D= \left[ \begin{array}{cc}a_{12}-\beta^2 m_1&-a_{12}\\-a_{21}&a_{21}-\beta^2 m_2- \IU  \beta d_2\end{array}\right],
\end{align*}
and $\det(L-\beta^2M- \IU  \beta D)=(a_{12}-\beta^2 m_1)(a_{21}-\beta^2 m_2- \IU  \beta d_2)-a_{12}a_{21}$. It is easy to see that the real part and imaginary parts of the determinant cannot be zero at the same time. Therefore, $ \mathcal{M}(\beta)$ is also nonsingular and a partially damped $2$-generator system cannot have any pure imaginary eigenvalues. 

Now, we prove the theorem for $n=3$. Let $A\in\mathbb{R}^{2n\times 2n}$. For index sets $\mathcal{I}_1\subseteq\{1,\cdots,2n\}$ and $\mathcal{I}_2\subseteq\{1,\cdots,2n\}$, we denote by $A[\mathcal{I}_1,\mathcal{I}_2]$ the (sub)matrix of entries that lie in the rows of $A$ indexed by $\mathcal{I}_1$ and the columns indexed by $\mathcal{I}_2$. For a $3$-generator system, the matrix $L$ can be written as
\begin{align*}
    L = \begin{bmatrix}
    a_{12}+a_{13} & -a_{12} & -a_{13}\\
    -a_{21} &  a_{21}+a_{23}  &-a_{23}\\
    -a_{31} & - a_{32} & a_{31} + a_{32}
    \end{bmatrix}
\end{align*}
where $a_{jk}\ge0, \forall j,k\in\{1,2,3\}, j\ne k$ and $a_{jk}=0 \iff a_{kj}=0$. Moreover, $M=\mathbf{diag}(m_1,m_2,m_3)$ and $D=\mathbf{diag}(0,d_2,d_3)$.
%
We complete the proof in three steps:
\begin{itemize}
\item Step $1$: We show that the first four columns of $\mathcal{M}(\beta)$ are linearly independent, i.e., $\rank(\mathcal{M}(\beta))\ge4$.\\
To do so, we show that the equation 
\begin{align*}
    \mathcal{M}(\beta)\left [ \{1,...,6\},\{1,2,3,4\}\right ] 
    \begin{bmatrix}
x_1 \\ x_2 \\x_3 \\ x_4
\end{bmatrix} = 0
\end{align*}
has only the trivial solution.

\begin{enumerate}[(i)]
	\item If $a_{12}+a_{13}-\beta^2 m_1\ne0$, then $x_4=0$. Moreover, we have $\beta d_2 x_2 = 0$ and $\beta d_3 x_3 = 0$ which imply $x_2=x_3=0$ because $\beta, d_2$, and $d_3$ are nonzero scalars. Finally, the connectivity assumption requires that at least one of the two entries $a_{21}$ and $a_{31}$ are nonzero, implying that $x_1=0$.
	
	\item If $a_{12}+a_{13}- \beta^2 m_1=0$, then by expanding the fifth and sixth rows we get
	\begin{align*}
	    & \beta d_2x_2 -a_{21}x_4=0  \implies x_2=\frac{a_{21}}{\beta d_2}x_4,\\
		& \beta d_3x_3 -a_{31}x_4=0, \implies x_3= \frac{a_{31}}{\beta d_3}x_4.
	\end{align*}
	Expanding the first row and substituting $x_2$ and $x_3$ from above gives
	\begin{align*}
	& -a_{12}x_2-a_{13}x_3=0 \implies -\frac{a_{12} a_{21}}{\beta d_2}x_4 -\frac{a_{13}a_{31}}{\beta d_3}x_4 =0.
	\end{align*}
	The connectivity assumption (and the fact that $a_{kj}\ge0, \forall k\ne j$ and $a_{kj}=0\iff a_{jk}=0$) leads to $x_4=0$. This implies $x_2=x_3=0$ and further $x_1=0$ due to the connectivity assumption.
	
\end{enumerate}

\item Step $2$: We prove that the first five columns of $\mathcal{M}(\beta)$ are linearly independent, i.e., $\rank(\mathcal{M}(\beta))\ge5$.\\
To do so, we show that the equation

\begin{align*}
    \mathcal{M}(\beta) \left [\{1,...,6\},\{1,2,3,4\} \right ] 
    \begin{bmatrix}
x_1 \\ x_2 \\x_3 \\ x_4
\end{bmatrix} = \begin{bmatrix}
0 \\ -\beta d_2 \\0 \\ -a_{12} \\ a_{21}+a_{23}- \beta^2 m_2 \\ -a_{32}
\end{bmatrix}
\end{align*}
has no solution, i.e., the fifth column is not in the span of the first four columns. Based on the equation in the fourth row we consider the following situations:
\begin{enumerate}[(i)]
	\item If $a_{12}+a_{13}- \beta^2 m_1=0$ and $a_{12}\ne0$, then there exists no solution.
	
	\item If $a_{12}+a_{13}- \beta^2 m_1=0$ and $a_{12}=0$, then $a_{13}=\beta^2 m_1$. Expanding the first row yields $-a_{13}x_3=0\implies x_3=0$. Expanding the second row provides $(a_{23}-\beta^2 m_2)x_2=-\beta d_2 \implies x_2=- \frac{\beta d_2}{(a_{23}- \beta^2 m_2)}$. Note that we assume $(a_{23}-\beta^2 m_2)\ne0$, since otherwise the system has no solution.
	Finally, we expand the fifth row and substitute $x_2$ into it:
	\begin{align*}
	 \beta d_2 x_2 = a_{23}-\beta^2 m_2  
	&  \implies -\frac{(\beta d_2)^2}{(a_{23}-\beta^2 m_2)} = a_{23}-\beta^2 m_2 \\ 
	&  \implies -(\beta d_2)^2 = (a_{23}-\beta^2 m_2)^2
	\end{align*}
	which is a contradiction.
	
	\item If $a_{12}+a_{13}-\beta^2 m_1\ne0$ and $a_{12}=0$, then $x_4=0$. By expanding the fifth and sixth rows we get
	\begin{align*}
	& \beta d_2x_2 = a_{23}- \beta^2 m_2  \implies x_2=\frac{a_{23}-\beta^2 m_2}{\beta d_2},\\
	& \beta d_3x_3 = -a_{32} , \implies x_3= -\frac{a_{32}}{\beta d_3}.
	\end{align*}
	Expanding the second row and substituting $x_2$ and $x_3$ from above gives
	\begin{align*}
	 (a_{23}-\beta^2 m_2)x_2-a_{23}x_3 = -\beta d_2 
	 \implies  \frac{(a_{23}-\beta^2 m_2)^2}{\beta d_2} + \frac{a_{23}a_{32}}{\beta d_3} = -\beta d_2
	\end{align*}
   which is a contradiction.
   
   \item If $a_{12}+a_{13}-\beta^2 m_1\ne0$ and $a_{12}\ne0$, then $x_4=\frac{-a_{12}}{a_{12}+a_{13}- \beta^2 m_1}$. By expanding the fifth and sixth rows and substituting $x_4$ we get
   \begin{align*}
   & \beta d_2x_2   + \frac{a_{12}a_{21}}{a_{12}+a_{13}- \beta^2 m_1} = a_{21}+a_{23}- \beta^2 m_2,\\
   & \beta d_3x_3   + \frac{a_{12}a_{31}}{a_{12}+a_{13}-\beta^2 m_1} = -a_{32}.
   \end{align*}
   Now we expand the first row to get $  x_1=\frac{a_{12}x_2+a_{13}x_3}{a_{12}+a_{13}-\beta^2 m_1}$. Finally, we expand the second row and substitute for $x_1, x_2$, and $x_3$:
   \begin{align*}
    -a_{21}\frac{a_{12}x_2+a_{13}x_3}{a_{12}+a_{13}-\beta^2 m_1} + (a_{21}+a_{23}-\beta^2 m_2)x_2 -a_{23}x_3=-\beta d_2, 
   \end{align*}
   which implies
   \begin{align*}
   & (  (a_{21}+a_{23}- \beta^2 m_2) -\frac{a_{12}a_{21}}{a_{12}+a_{13}-\beta^2 m_1} ) x_2  \\
   & - (a_{23}  +\frac{a_{13}a_{21}}{a_{12}+a_{13}-\beta^2 m_1}) x_3 =-\beta d_2, 
   \end{align*}
   or equivalently
   \begin{align*}
   &  \frac{1}{\beta d_2} (  (a_{21}+a_{23}- \beta^2 m_2) -\frac{a_{12}a_{21}}{a_{12}+a_{13}- \beta^2 m_1} )^2 \\
   & + \frac{1}{\beta d_3}  (a_{23}  +\frac{a_{13}a_{21}}{a_{12}+a_{13}- \beta^2 m_1})^2 =-\beta d_2.
   \end{align*}
   which is a contradiction.
\end{enumerate}

\item Step $3$: $\rank(\mathcal{M}(\beta))$ is an even number.\\
Finally, \Cref{lemma: complex representation} precludes the rank of $\mathcal{M}(\beta)$ from being equal to $5$. Therefore, $\rank(\mathcal{M}(\beta))=6$, i.e., $\mathcal{M}(\beta)$ is always nonsingular. This completes the proof.
\end{itemize}

\end{document}